\documentclass[a4paper,reqno]{amsart}
\usepackage[english]{babel}
\usepackage{amsmath,amsthm,amssymb,amsfonts}

\usepackage{enumitem}

\usepackage{hyperref}

\newtheorem{thm}{Theorem}[section]
\newtheorem{prp}[thm]{Proposition}
\newtheorem{cor}[thm]{Corollary}
\newtheorem{lem}[thm]{Lemma}
\theoremstyle{definition}
\newtheorem{dfn}[thm]{Definition}
\newtheorem{rem}[thm]{Remark}

\numberwithin{equation}{section}

\newcommand{\RR}{\mathbb{R}}
\newcommand{\CC}{\mathbb{C}}
\newcommand{\NN}{\mathbb{N}}
\newcommand{\ZZ}{\mathbb{Z}}

\newcommand{\Npos}{\mathbb{N}_+}
\newcommand{\Rpos}{\mathbb{R}^+}
\newcommand{\Rnon}{\mathbb{R}^+_0}

\newcommand{\tc}{\,:\,}
\newcommand{\defeq}{:=}
\newcommand{\eqdef}{=:}

\newcommand{\BigO}{\mathcal{O}}

\DeclareMathOperator{\Span}{span}

\newcommand{\sobolev}[2]{L^{#1}_{#2}}

\DeclareMathOperator*{\esssup}{ess\,sup}

\DeclareMathOperator{\supp}{supp}

\newcommand{\dist}{\varrho}
\newcommand{\weight}{\varpi}
\newcommand{\Vol}{\mathcal{V}}

\newcommand{\Mult}{W}

\newcommand{\vu}{u}
\newcommand{\uK}{K_0}

\newcommand{\Kern}{\mathcal{K}}
\newcommand{\opL}{\mathcal{L}}
\newcommand{\opH}{\mathfrak{H}}

\newcommand{\done}{{d_1}}
\newcommand{\dtwo}{{d_2}}
\newcommand{\jone}{j}
\newcommand{\jtwo}{k}

\newcommand{\pdeg}{d}
\newcommand{\hpdeg}{\sigma}

\newcommand{\scale}{\xi}
\newcommand{\FScale}{\Xi}
\newcommand{\ESum}{\Sigma}

\newcommand{\Pot}{\mathcal{P}}
\newcommand{\MPot}{\mathcal{MP}}
\newcommand{\GPot}{\mathfrak{P}}
\newcommand{\co}{\varkappa}

\newcommand{\even}{\mathrm{e}}
\newcommand{\convex}{\mathrm{c}}

\newcommand{\zosc}{\mathrm{osc}}
\newcommand{\ztrans}{\mathrm{trans}}
\newcommand{\zexp}{\mathrm{exp}}

\newcommand{\Ai}{\mathrm{Ai}}   
\newcommand{\Bi}{\mathrm{Bi}}

\DeclareMathOperator{\Four}{\mathcal{F}}   

\DeclareFontFamily{U}{matha}{\hyphenchar\font45}
\DeclareFontShape{U}{matha}{m}{n}{
      <5> <6> <7> <8> <9> <10> gen * matha
      <10.95> matha10 <12> <14.4> <17.28> <20.74> <24.88> matha12
      }{}
\DeclareSymbolFont{matha}{U}{matha}{m}{n}
\DeclareFontSubstitution{U}{matha}{m}{n}

\DeclareFontFamily{U}{mathx}{\hyphenchar\font45}
\DeclareFontShape{U}{mathx}{m}{n}{
      <5> <6> <7> <8> <9> <10>
      <10.95> <12> <14.4> <17.28> <20.74> <24.88>
      mathx10
      }{}
\DeclareSymbolFont{mathx}{U}{mathx}{m}{n}
\DeclareFontSubstitution{U}{mathx}{m}{n}

\DeclareMathDelimiter{\vvvert}{0}{matha}{"7E}{mathx}{"17}

\begin{document}

\title[Multiplier theorem for Grushin operators]{A robust approach to sharp multiplier theorems for Grushin operators}
\author{Gian Maria Dall'Ara}
\address[G.\ M.\ Dall'Ara]{Fakult\"at f\"ur Mathematik \\ Oskar--Morgenstern--Platz 1 \\ 1090 Vienna \\ Austria}
\email{gianmaria.dallara@univie.ac.at}
\author{Alessio Martini}
\address[A.\ Martini]{School of Mathematics \\ University of Birmingham \\ Edgbaston \\ Birmingham \\ B15 2TT \\ United Kingdom}
\email{a.martini@bham.ac.uk}

\thanks{
The first-named author was supported by the FWF-project P28154, while the second-named author was supported in part by the EPSRC Grant ``Sub-Elliptic Harmonic Analysis'' (EP/P002447/1). Part of this work was developed during a visit of the second-named author to the Mathematisches Seminar of the Christian--Albrechts--Universit\"at zu Kiel (Germany), made possible by the University's kind hospitality and the financial support of the Alexander von Humboldt Foundation.}


\begin{abstract}
We prove a multiplier theorem of Mihlin--H\"ormander type for operators of the form $-\Delta_x - V(x) \Delta_y$ on $\RR^\done_x \times \RR^\dtwo_y$, where $V(x) = \sum_{\jone=1}^\done V_\jone(x_\jone)$, the $V_\jone$ are perturbations of the power law $t \mapsto |t|^{2\hpdeg}$, and $\hpdeg \in (1/2,\infty)$. The result is sharp whenever $\done \geq \hpdeg \dtwo$. The main novelty of the result resides in its robustness: this appears to be the first sharp multiplier theorem for nonelliptic subelliptic operators allowing for step higher than two and perturbation of the coefficients.
 The proof hinges on precise estimates for eigenvalues and eigenfunctions of one-dimensional Schr\"odinger operators, which are stable under perturbations of the potential.
\end{abstract}


\maketitle

\section{Introduction}

\subsection{Setting and main result} Let $X$ be a measure space and $\opL$ a self-adjoint operator on $L^2(X)$. A Borel functional calculus for $\opL$ is defined via the spectral theorem and, for all bounded Borel functions $F : \RR \to \CC$, the operator
\[
F(\opL) = \int_\RR F(\lambda) \,dE(\lambda)
\]
is bounded on $L^2(X)$ (here $E$ is the spectral resolution of $\opL$). Boundedness of the ``spectral multiplier'' $F$ is in general not enough, however, to guarantee $L^p$-boundedness for $p \neq 2$ of the operator $F(\opL)$.

In the case $\opL = -\Delta$ is the Laplace operator on $\RR^n$, $L^p$-boundedness of $F(\opL)$ is related to smoothness properties of $F$. The Mihlin--H\"ormander multiplier theorem \cite{mihlin_multipliers_1956,hrmander_estimates_1960} indeed implies that $F(\opL)$ is of weak type $(1,1)$ and $L^p$-bounded for all $p \in (1,\infty)$ whenever $F$ satisfies the local scale-invariant smoothness condition
\begin{equation}\label{eq:mhcond}
\sup_{t > 0} \| F(t \cdot) \, \chi \|_{\sobolev{q}{s}(\RR)} < \infty
\end{equation}
for $q=2$ and some $s > n/2$; here $\sobolev{q}{s}(\RR)$ denotes the $L^q$ Sobolev space of fractional order $s$ and $\chi \in C^\infty_c((0,\infty))$ is any nontrivial cutoff. Strong $L^1$ boundedness of $F(\opL)$ does not hold in general under this assumption, but can be recovered, e.g., when $F$ is compactly supported and belongs to $\sobolev{2}{s}(\RR)$ for some $s > n/2$: this corresponds, e.g., to the $L^1$-boundedness of the Bochner--Riesz means $(1-t\opL)_+^\alpha$ whenever $\alpha > (n-1)/2$ and $t > 0$.

The smoothness condition $s>n/2$ in these results is sharp, in the sense that $n/2$ cannot be replaced by a smaller quantity (see, e.g., \cite{sikora_imaginary_2001} and references therein). In addition, the validity of these results has little to do with the symmetries of the Euclidean Laplace operator (such as homogeneity and translation-invariance): indeed analogous sharp results can be obtained in the case where $\opL$ is an elliptic self-adjoint (pseudo)differential operator on a compact manifold \cite{seeger_boundedness_1989}. Weakening the ellipticity assumption on $\opL$, on the other hand, turns out to be a more delicate issue and obtaining sharp multiplier theorems of Mihlin--H\"ormander type for nonelliptic operators $\opL$ is in general a challenging and widely open problem.

Interesting classes of nonelliptic differential operators with polynomial coefficients on Euclidean spaces were introduced in \cite{grushin_certain_1970}, including operators on $\RR^{\done}_x \times \RR^{\dtwo}_y$ of the form
\begin{equation}\label{eq:grushin_intro}
\opL= -\Delta_x - V(x) \Delta_y,
\end{equation}
where $\Delta_x = \sum_{\jone=1}^\done \partial_{x_\jone}^2$ and $\Delta_y = \sum_{\jtwo=1}^\dtwo \partial_{y_\jtwo}^2$ are the two ``partial Laplacians'' in $x$ and $y$, and $V(x) = |x|^{2\hpdeg}$ for some $\hpdeg \in \NN$.
If $\hpdeg > 0$, then the operator $\opL$ 
 is a ``degenerate elliptic operator'', in the sense that it is elliptic off the ``singular region'' $\{ (x,y) \tc x = 0\}$ where the coefficient $V(x)$ vanishes. Nevertheless the operator $\opL$ is hypoelliptic and satisfies subelliptic estimates: this follows from a celebrated result of H\"ormander's \cite{hrmander_hypoelliptic_1967}, since one can write $\opL$ as (minus) the sum of squares of a system of vector fields satisfying H\"ormander's bracket-generating condition (see, e.g., the discussion in \cite[Section 3]{robinson_grushin_2016} for details).
Moreover, this same condition allows one to associate with the Grushin operator $\opL$ a sub-Riemannian geometric structure (see, e.g., \cite{nagel_balls_1985,montgomery_tour_2002}). The corresponding Carnot--Carath\'eodory distance $\dist$ on $\RR^\done \times \RR^\dtwo$  satisfies the doubling condition
\[
\Vol(z,\lambda R) \leq C \, \lambda^Q \, \Vol(z,R)
\]
for some $C > 0$ and all $z \in \RR^\done \times \RR^\dtwo$, $R > 0$ and $\lambda \geq 1$; here $\Vol(z,R)$ denotes the (Lebesgue) measure of the $\dist$-ball of centre $z$ and radius $R$, while
\[
Q = \done + (1+\hpdeg) \dtwo
\]
is the so-called ``homogeneous dimension''. In addition, the Grushin operator $\opL$ satisfies Gaussian-type heat kernel bounds, as well as finite propagation speed for the corresponding wave equation, relative to the distance $\dist$. All these properties are indeed shown in \cite{robinson_analysis_2008} for a broad class of Grushin-type operators, including operators of the form \eqref{eq:grushin_intro} where the coefficient $V$ is only assumed to be a nonnegative measurable function such that
\begin{equation}\label{eq:rs_cond}
\co^{-1} |x|^{2\hpdeg} \leq V(x) \leq \co |x|^{2\hpdeg}
\end{equation}
for some constant $\co \geq 1$ and a (possibly fractional) exponent $\hpdeg \in (0,\infty)$. As a consequence, due to general results of \cite{hebisch_functional_1995,cowling_spectral_2001,duong_plancherel-type_2002}, a multiplier theorem of Mihlin--H\"ormander type holds for $\opL$, yielding weak type $(1,1)$ and $L^p$-boundedness for $p\in(1,\infty)$ of an operator of the form $F(\opL)$ whenever the condition \eqref{eq:mhcond} with $q=\infty$ and some $s > Q/2$ is satisfied; correspondingly, $L^1$-boundedness of Bochner--Riesz means $(1-t\opL)_+^\alpha$ is obtained whenever $\alpha > Q/2$ \cite[Section 8.2]{robinson_analysis_2008}.

The smoothness condition $s>Q/2$ may appear as the natural analogue of the condition $s>n/2$ for the Laplace operator on $\RR^n$ (or a more general elliptic operator on an $n$-manifold): indeed, the homogeneous dimension $Q$ is a natural dimensional parameter for the geometry associated with the Grushin operator $\opL$. Differently from the elliptic case, however, the condition $s > Q/2$ need not be sharp. The mismatch between the homogeneous dimension and the ``sharp Mihlin--H\"ormander threshold'' for a nonelliptic subelliptic operator was first discovered in the case of a homogeneous left-invariant sub-Laplacian on a Heisenberg group \cite{hebisch_multiplier_1993,mueller_spectral_1994}; in that case it was shown that the condition \eqref{eq:mhcond} with $q=2$ and $s>d/2$ is enough (and sharp), where $d$ is the topological dimension of the group.  After that discovery, a number of results were obtained for subelliptic operators in a variety of settings (and especially in the case of homogeneous sub-Laplacians on $2$-step stratified groups), improving on the condition $s>Q/2$ and often showing that $s>d/2$ is enough (see, e.g., the discussion in \cite{martini_necessary_2016}).

It should be noted that, in the case $V(x) = |x|^{2\hpdeg}$ with $\hpdeg \in \NN$, the Grushin operator $\opL$ defined in \eqref{eq:grushin_intro} can be lifted to a homogeneous left-invariant sub-Laplacian on a stratified group of step $\hpdeg+1$, and a number of properties of $\opL$ can be deduced from the analysis of the corresponding sub-Laplacian \cite{rothschild_hypoelliptic_1976}. In particular, when $\hpdeg = 1$ and $\dtwo = 1$, the Grushin operator $\opL$ corresponds to a sub-Laplacian on a Heisenberg group. Note that the lifting procedure increases the dimension of the underlying space, hence a sharp multiplier theorem for the sub-Laplacian need not directly imply a sharp result for the corresponding Grushin operator. Nevertheless the mentioned results for Heisenberg and related groups make it plausible that the general multiplier theorem for Grushin operators of \cite{robinson_analysis_2008} may be improved.

Indeed in \cite{martini_grushin_2012,martini_sharp_2014} the case $V(x) = |x|^2$ was treated for all values of $\done$ and $\dtwo$, proving that the condition \eqref{eq:mhcond}
for $q=2$ and some $s > (\done+\dtwo)/2$ is enough to guarantee the weak-type $(1,1)$ and $L^p$-boundedness for $p\in(1,\infty)$ of $F(\opL)$. Note that $\done+\dtwo$ is the topological dimension of $\RR^\done \times \RR^\dtwo$. A contraction argument \cite{mitjagin_divergenz_1974,kenig_divergence_1982} exploiting the ellipticity of $\opL$ off the singular region can be used to show that the condition $s > (\done+\dtwo)/2$ is sharp.
In addition, in \cite{chen_sharp_2013} the case $V(x) = \sum_{\jone=1}^\done |x_\jone|$ was considered, and a multiplier theorem with condition \eqref{eq:mhcond} for $q=2$ and some $s > \max\{\done+\dtwo,3\dtwo/2\}/2$ was proved. By the contraction argument cited above, this result is sharp when $\done\geq\dtwo/2$.

Our main theorem is a significant generalization of these results. In order to state it, let us introduce the class $\Pot^{2\hpdeg}_{\convex\even}$ of convex even functions $U : \RR \to \RR$ which are of class $C^3$ away from the origin and satisfy the inequalities
\begin{align*}
\co^{-1} t^{2\hpdeg} \leq U(t) &\leq \co t^{2\hpdeg}, \\
\co^{-1} t^{2\hpdeg-1} \leq U'(t) &\leq \co t^{2\hpdeg-1}, \\
|U''(t)| &\leq \co t^{2\hpdeg-2}, \\
|U'''(t)| &\leq \co t^{2\hpdeg-3}
\end{align*}
for some constant $\co \geq 1$ and all $t>0$. Clearly the first of the above inequalities is analogous to \eqref{eq:rs_cond}; here in addition we require a control of the derivatives of $U$ with the corresponding derivatives of $|\cdot|^{2\hpdeg}$ up to order $3$. Here is our result.

\begin{thm}\label{thm:main}
Let $\hpdeg \in (1/2,\infty)$. Let $\opL$ be defined by \eqref{eq:grushin_intro}, where $V : \RR^\done \to \RR$ can be written as
\begin{equation}\label{eq:potential_sum}
V(x) = \sum_{\jone=1}^\done V_\jone(x_\jone),
\end{equation}
and $V_1,\dots,V_\done \in \Pot_{\convex\even}^{2\hpdeg}$. Let $D = \max\{\done+\dtwo,(1+\hpdeg)\dtwo\}$. Let $F : \RR \to \CC$ be a bounded Borel function. Then the following hold.
\begin{enumerate}[label=(\roman*)]
\item If $\supp F \subseteq [1/2,2]$ and $F \in \sobolev{2}{s}(\RR)$ for some $s > D/2$, then
\[
\sup_{t>0} \| F(t\opL) \|_{1 \to 1} \leq C_s \|F\|_{\sobolev{2}{s}(\RR)}.
\] 
\item If \eqref{eq:mhcond} is satisfied for $q=2$ and some $s > D/2$, then $F(\opL)$ is of weak-type $(1,1)$ and bounded on $L^p(\RR^\done \times \RR^\dtwo)$ for all $p \in (1,\infty)$, and moreover
\begin{align*}
\| F(\opL)\|_{p \to p} &\leq C_{s,p} \sup_{t>0} \|F(t \cdot) \, \chi\|_{\sobolev{2}{s}(\RR)},\\
\| F(\opL)\|_{L^1 \to L^{1,\infty}} &\leq C_{s} \sup_{t>0} \|F(t \cdot) \, \chi\|_{\sobolev{2}{s}(\RR)}.
\end{align*}
\item The Bochner--Riesz means $(1-t\opL)_+^\alpha$ are $L^p$-bounded for all $p\in [1,\infty]$ uniformly in $t>0$ whenever $\alpha > (D-1)/2$.
\end{enumerate}
\end{thm}

A few comments may help to clarify the scope of Theorem \ref{thm:main}.
\begin{enumerate}
\item The parameter $D$ is strictly less than the homogeneous dimension $Q$ for all values of $\done$ and $\dtwo$, and moreover our smoothness condition is expressed in terms of an $L^q$ Sobolev norm with $q=2$ instead of $q=\infty$; hence, when it applies, our result always yields an improvement to the general theorem of \cite{robinson_analysis_2008}.
\item We can treat the case $V(x) = \sum_{\jone=1}^\done |x_\jone|^{2\hpdeg}$ for all $\hpdeg \in (1/2,\infty)$. In particular we ``interpolate'' between the previously known results of \cite{martini_grushin_2012,martini_sharp_2014,chen_sharp_2013}, corresponding to $\sigma=1/2$ and $\sigma=1$, thus answering a question posed in \cite{chen_sharp_2013}. 
\item\label{en:scope_sharpness}  By the aforementioned contraction argument, our result is sharp whenever $\done \geq \hpdeg \dtwo$ (so that $D = \done+\dtwo$ is the topological dimension). This means that, for all values of $\hpdeg \in (1/2,\infty)$, we obtain a sharp multiplier theorem (with a suitable choice of $\done$ and $\dtwo$).
\item\label{en:scope_perturbation} We do not assume that $V$ is either algebraic or homogeneous: each of the summands in \eqref{eq:potential_sum} can be perturbed on $\Rpos$ in a scale-invariant $C^3$ fashion; in this sense, our result is ``perturbation-invariant''.
\end{enumerate}

We believe that points (\ref{en:scope_sharpness}) and (\ref{en:scope_perturbation}) above reveal the main significance and interest of our result: indeed, to the best of our knowledge, the previously known sharp multiplier theorems for sub-Laplacians and related subelliptic operators have only been obtained in the case of step at most $2$ (corresponding to $\sigma \leq 1$), and only for quite rigid classes of operators (with algebraic or analytic coefficients, and possessing a number of symmetries). While our assumptions on the operator $\opL$ are still somewhat restrictive, when compared to the assumptions in \cite{robinson_analysis_2008}, nevertheless they appear to be a substantial relaxation of those in the previously known sharp results; in these respects, our result may be considered as a step forward in the investigation of the ``sharp Mihlin--H\"ormander threshold'' for general subelliptic operators.


\subsection{Main ingredients of the proof and ties with mathematical physics}\label{sec:novel}
The main advantage of working with Grushin operators is that, exploiting their peculiar structure, precise information on their spectral theory and functional calculus can be obtained from the analysis of certain
families of Schr\"odinger operators.

Namely, if $\opL$ is as in \eqref{eq:grushin_intro} and $\Four f$ denotes the partial Fourier transform of $f \in L^2(\RR^\done_x \times \RR^\dtwo_y)$ in the variable $y$, then
\[
\Four \opL f(x,\eta) = \opL_{|\eta|^2} \Four f(x,\eta),
\]
where $\eta$ is the dual variable to $y$, and, for all $\scale \in (0,+\infty)$, $\opL_\scale$ is the Schr\"odinger operator on $\RR^\done_x$ given by
\[
\opL_\scale = -\Delta_x + \scale V(x).
\]
Under the ``decomposability assumption'' \eqref{eq:potential_sum}, we can further write $\opL_\scale$ as the sum of the one-dimensional Schr\"odinger operators
\[\label{intro-schrod}
\opL_{\jone,\scale} = -\partial_{x_\jone}^2 + \scale V_\jone(x_\jone)
\]
for $\jone=1,\dots,\done$, each of which acts on a different variable $x_\jone$. In the present paper, as well as in \cite{martini_grushin_2012,martini_sharp_2014,chen_sharp_2013}, a detailed analysis of these families of operators (for appropriate choices of $V_\jone$) allows one to prove certain ``weighted Plancherel estimates'' from which the corresponding sharp multiplier theorems follow by well-developed techniques. More precisely, proving weighted Plancherel estimates boils down to bounding from above certain sums involving eigenvalues and eigenfunctions of the operators $\opL_{\jone,\scale}$. 

In the case where $V$ is homogeneous, a substantial simplification occurs: indeed the operators $\opL_{\jone,\scale}$ are conjugate to one another via suitable scalings of the variable $x_\jone$, hence the problem reduces to the analysis of a single Schr\"odinger operator $\opL_{\jone,1}$ for each $\jone=1,\dots,\done$. The previously known sharp results for Grushin operators fall into this class: indeed, \cite{martini_grushin_2012,martini_sharp_2014} are based on the analysis of the ``harmonic oscillator'' $-\partial_t^2 + t^2$, while \cite{chen_sharp_2013} is based on the analysis of the ``anharmonic oscillator'' $-\partial_t^2 + |t|$. Moreover, in these particular cases, the eigenfunctions can be expressed in terms of special functions (namely, Hermite polynomials and the Airy function), for which a number of estimates are readily available in the literature.

When $V$ is not homogeneous, there seems to be no way to directly relate the spectral decompositions of the various $\opL_{\jone,\scale}$. Moreover, in the generality of the class of potentials $\Pot^{2\hpdeg}_{\even\convex}$, one cannot obtain exact expressions for eigenfunctions in terms of already-studied ``special functions''. Hence we have to work simultaneously with all the different Schr\"odinger operators $\opL_{\jone,\scale}$ and look for estimates with a suitable uniformity in the parameter $\scale$. In particular we need a precise understanding of the behaviour of eigenfunctions in the so-called \emph{semiclassical regime} $\scale\rightarrow+\infty$. A mathematical physics tool devised to deal with such a problem is the \emph{WKB approximation} (see, e.g., Chapter $15$ of \cite{hall-book} for an introduction), which, among other things, allows one to understand the behaviour of eigenfunctions in the most delicate region, that is, around the transition points that separate the ``classical region'', where the potential is smaller than the energy level, and its complement. A key role in this approximation is played by the Airy function. Despite the effectiveness of this procedure, it is far from obvious how to derive from it estimates that possess the kind of uniformity in the space variable, the energy and the values of the parameter $\scale$ that we need for our purposes.

In the present paper we follow a different route, resorting to a general method due to Olver \cite{olverbook} to obtain approximate representations of solutions to second-order ODEs with a ``simple turning point''. Thanks to this tool we obtain, after some work, estimates with the desired uniformities. Perhaps not surprisingly, Olver's representation of solutions involves again Airy functions.

Although Olver's method plays the central role in our analysis, a number of technical problems arise when one tries to derive our main result from the estimates of eigenfunctions one gets out of it. Here we would like to point out a few of the tools that we employ to deal with these difficulties:\begin{enumerate}
\item a uniform version (due to Hartman and Titchmarsh) of the so-called ``Bohr--Sommerfeld formula'', which we use in a somewhat unconventional way to estimate ``gaps'' between transition points corresponding to different energy levels and values of $\scale$; indeed, in the case $\done > 1$, we are led to studying the separation of lattice-like structures formed by ``vectors of transition points'' corresponding to the different Schr\"odinger operators $\opL_{\jone,\scale}$;
\item the $L^2$-boundedness of ``Riesz transforms of arbitrary order'' associated with one-dimensional Schr\"odinger operators with potentials in our class $\Pot^{2\hpdeg}_{\convex\even}$; while this appears to be known in particular cases, such as that of polynomial potentials (where the problem can be reduced via lifting to subelliptic estimates for homogeneous sub-Laplacians on stratified groups), a corresponding result in the generality we need does not seem to exist in the literature;
\item a virial-type bound for eigenfunctions (again for our class of potentials), which asserts that a significant fraction of the total energy comes from the potential term or, equivalently, that potential and kinetic energy of ``eigenstates'' are comparable (see, e.g., \cite{fock,weidmann,georgescu_virial_1999} for the classical virial theorem in quantum mechanics).
\end{enumerate}


\subsection{Open questions}

As already mentioned, our result is certainly not the definitive answer to the problem of obtaining sharp multiplier theorems for Grushin operators (or more general subelliptic operators), and a number of questions remain open. We would like to list a few of them.
\begin{enumerate}

\item We know that our result is sharp in the case $\done \geq \hpdeg \dtwo$, where the parameter $D$ coincides with the topological dimension $\done+\dtwo$. Can the result be improved when $\done < \hpdeg \dtwo$, and $D$ replaced with $\done+\dtwo$ in any case? The methods used in this paper, based on ``weighted Plancherel estimates'' with weights ``depending only on the variable $x$'', appear not to be suitable to obtain such a result when $\done \ll \dtwo$. On the other hand, from \cite{martini_sharp_2014} we know that this improvement is possible in the particular case of the ``harmonic oscillator'' Grushin operator. Namely, in \cite{martini_sharp_2014} a different method (with weights ``depending also on the variable $y$'') is developed, which however is based on special identities for Hermite polynomials; a challenging problem is whether a more robust version of this method can be applied in the generality of our assumptions.

\item Another question that might be investigated is whether the restriction $\hpdeg > 1/2$ is really necessary. Recall that the ``unperturbed case'' with $\hpdeg =1/2$, treated in \cite{chen_sharp_2013}, was based on the analysis of the anharmonic operator, whose eigenfunctions are expressed in terms of the Airy function. The fact that Olver's method gives an approximate expression of solutions to ODEs in terms of the Airy function somehow explains our restriction on $\hpdeg$: in order to obtain a uniform control of the error, we need a better local behaviour of the perturbed potential and its derivatives compared to the ``approximating potential'' $|\cdot|$ corresponding to the Airy function. However one may wonder whether there exist alternative methods that allow one to ``interpolate'' between our results and the elliptic case $\hpdeg = 0$.

\item Despite the ``perturbation-invariant'' character of our assumptions, the ``decomposability condition'' \eqref{eq:potential_sum} appears still to be a very strong structural assumption on $V$. It would be interesting to know whether it is possible to get rid of this assumption and consider ``genuinely multi-dimensional'' potentials $V$. This would clearly require changing considerably the techniques exploited in this work: for example, in trying to follow the approach sketched above, one would need to get precise estimates for eigenfunctions of multi-dimensional Schr\"odinger operators, a problem that is considerably harder than its one-dimensional analogue.

\item Finally, one could consider Grushin operators of the form \eqref{eq:grushin_intro} on $\RR^\done \times \RR^\dtwo$ as prototypes of more general degenerate elliptic operators on manifolds and ask whether similar sharp results can be obtained in this greater generality. In this vein, the simplest example one could think of is probably a sum-of-squares operator $\opL = -(X^2+Y^2)$ on a $2$-dimensional compact manifold, where the vector fields $X$ and $Y$ are allowed to vanish, but together with their commutator $[X,Y]$ span the tangent at each point: already for this apparently simple example the so-far available techniques do not appear to be enough to obtain a sharp multiplier theorem for $\opL$. The recent result \cite{casarino_grushinsphere}, devoted to the analysis of a particular Grushin-type operator on the $2$-sphere in $\RR^3$, indicates the possibility of treating such operators on compact manifolds; however the problem of obtaining a ``perturbation-invariant'' result in this context remains open.
\end{enumerate}

\subsection{Structure of the paper} 
\par In Section \ref{s:olver} we recall Olver's result on approximate solutions to ODEs with a simple turning point, and show how it can be used to prove estimates for square-integrable solutions on $\Rpos$ of the equation $u''=\alpha^2(U-1)u$, when $\alpha$ is a large positive parameter and $U$ is in one of our classes of potentials $\Pot^{2\hpdeg}_{\even\convex}$ (in fact, many results of this and the following sections hold in slightly greater generality). 
\par Next, in Section \ref{s:schroedinger} we prove a few results for one-dimensional Schr\"odinger operators with potentials in our classes: $L^2$-bounds of Riesz transforms of arbitrary order (Proposition \ref{apriori_scaled-prp}), the virial-type integral bound (Proposition \ref{virial-prp}), and precise pointwise estimates for eigenfunctions (Proposition \ref{eigenfunctions-prp}). Except for the bound on Riesz transforms, these proofs rely heavily on the estimates of Section \ref{s:olver}.
\par In Section \ref{s:rescaledschroedinger} we proceed to consider one-parameter families of Schr\"odinger operators of the form $-\partial_x^2+\scale V$ (again with $V$ in one of our classes). We prove bounds on eigenvalues, transition points, and derivatives with respect to $\scale$ of the eigenvalues (Proposition \ref{eigenvalues-prp}). We also state the Bohr--Sommerfeld formula with uniform error (Theorem \ref{titchmarsh-thm}). For the sake of completeness, we devote an Appendix to describing how the statement of Theorem \ref{titchmarsh-thm} follows from the arguments of \cite{titchmarsh}.
\par Finally, in Section \ref{s:grushin} we discuss the spectral theory and functional calculus of Grushin operators and we prove the weighted Plancherel estimates. The key steps are Lemma \ref{lem:transition_est} (where gaps between transition points are studied, crucially relying on the Bohr--Sommerfeld formula) and Proposition \ref{keytoweighted-prp} (where a pointwise estimate for the density of the ``Plancherel measure'' associated to Grushin operators is obtained). Finally, in Section \ref{ss:main}, we derive our main result, Theorem \ref{thm:main}.

\subsection{Notation}
We denote by $\NN$ (resp.\ $\Npos$, $\Rnon$, $\Rpos$) the set of natural numbers (resp. positive integers, nonnegative real numbers, positive real numbers).

\section{Analysis of $L^2$ solutions of $u''=\alpha^2(U-1)u$}\label{s:olver}

We begin with defining a class of functions on $\Rpos$ related to the class of potentials on $\RR$ featuring in Theorem \ref{thm:main}.

\begin{dfn}\label{pot-dfn}
For $\co,\pdeg \in \Rpos$, we denote by $\Pot^\pdeg_+(\co)$ the collection of $C^3$ functions $U : \Rpos \to \Rpos$ which satisfy
\begin{gather}
\co^{-1} x^{\pdeg} \leq U(x) \leq \co x^{\pdeg},\\
\co^{-1} x^{\pdeg-1} \leq U'(x) \leq \co x^{\pdeg-1},\\
|U''(x)| \leq \co x^{\pdeg-2}, \\
|U'''(x)| \leq \co x^{\pdeg-3}
\end{gather}
for all $x \in \Rpos$.
\end{dfn}

Since our estimates will depend on $U\in \Pot^\pdeg_+(\co)$ only through the parameters $\pdeg$ and $\co$, it is convenient to introduce the notation $A\lesssim B$ for the inequality $A\leq C B$, where $C$ is a positive constant that depends only on $\co$ and $\pdeg$. Accordingly, we write $A\simeq B$ when both $A\lesssim B$ and $B\lesssim A$ hold.

The goal of this section is to prove the following proposition, which will be crucial for the rest of our analysis.

\begin{prp}\label{trans-prp}
Let $\pdeg > 1$, $\co \in \Rpos$, and $U\in \Pot^\pdeg_+(\co)$. There exists $\alpha_0\simeq1$ such that if $\alpha\geq\alpha_0$ and $u$ is a solution of
\begin{equation}\label{ode-3}
u''(x)=\alpha^2  \, (U(x)-1) \, u(x) \qquad \forall x\in\Rpos
\end{equation}
such that $\int_{\Rpos} u^2<+\infty$, then the following pointwise estimates hold:
\begin{equation}\label{eq:olver_trans}
u(x)^2\lesssim|x-x_0|^{-1/2}\int_{\Rpos} u^2\qquad\forall x\in\Rpos
\end{equation}
and
\begin{equation}\label{eq:olver_unif}
u(x)^2\lesssim \alpha^{1/3} \int_{\Rpos} u^2\qquad\forall x\in\Rpos,
\end{equation}
where $x_0 \in \Rpos$ is uniquely defined by $U(x_0)=1$. Moreover, the following integral estimate holds:
\begin{equation}\label{eq:olver_int}
\int_{\Rpos}Uu^2\gtrsim \int_{\Rpos}u^2.
\end{equation}
\end{prp}

For the rest of the section, we work with a fixed $U\in \Pot^\pdeg_+(\co)$, focusing on the uniformity of our estimates, that is, on the dependence of the implicit constants only on $\pdeg$ and $\co$. 

\subsection{Olver's approximate solutions of ODEs with one simple turning point}\label{olver-sec}

From Definition \ref{pot-dfn} it follows immediately that there exists a unique point $x_0 \in \Rpos$ where $U(x_0)=1$, and moreover $U'(x_0)\neq 0$. In the classical language of ODEs, the equation \eqref{ode-3} has a \emph{simple turning point}. Following \cite{olverbook} (in particular, see Section $3$ of Chapter $11$), we introduce the new independent variable $\zeta$, related to $x$ as follows:
\[
\zeta(x) = \begin{cases}
-\left(\frac{3}{2}\int_x^{x_0}\sqrt{1-U}\right)^{2/3},  &x\in(0,x_0],\\
\left(\frac{3}{2}\int_{x_0}^x\sqrt{U-1}\right)^{2/3},   &x\in[x_0,+\infty).
\end{cases}
\]
It is easily seen that $\zeta$ is a homeomorphism of $\Rpos$ onto $(-b, +\infty)$ and that it is $C^4$ on $\Rpos\setminus\{x_0\}$, where
\[
b=\left(\frac{3}{2}\int_0^{x_0}\sqrt{1-U}\right)^{2/3} \in \Rpos.
\] 
It is proved in \cite{olverbook} that $\zeta$ is in fact $C^3$ on the whole $\Rpos$, and we will have more to say about that in what follows.

We next recall the definition of \emph{Olver's auxiliary functions} $E$ and $M$. Let $c$ be the negative root of $\Ai(x)=\Bi(x)$ of smallest absolute value, where $\Ai$ and $\Bi$ are the Airy function of first and second kind respectively. Then
\[ 
E(x)=\begin{cases}1,  &x\leq c,\\
\sqrt{\Bi(x)/\Ai(x)}, &x\geq c,\end{cases}
\]
while
\[
M(x)=\begin{cases}\sqrt{\Ai(x)^2+\Bi(x)^2}, &x\leq c,\\
\sqrt{2\Ai(x)\cdot \Bi(x)}, & x\geq c.\end{cases}
\]

The last ingredient we need to state Olver's theorem is the \emph{Schwarzian derivative}
\begin{equation}\label{schwarzian}
\Phi(\zeta) \defeq \left(\frac{dx}{d\zeta}\right)^{1/2} \frac{d^2}{d\zeta^2} \left[\left(\frac{dx}{d\zeta}\right)^{-1/2}\right]\qquad(\zeta\in (-b,+\infty)).
\end{equation}

Here is finally Olver's result \cite[Chapter 11, Theorem 3.1]{olverbook}.

\begin{thm}\label{olver-thm}
If
\begin{equation}\label{eq:J}
J \defeq \int_{-b}^{+\infty}\frac{|\Phi(\zeta)|}{\sqrt{|\zeta|}} \,d\zeta <+\infty,
\end{equation}
then for all $\alpha \in \Rpos$ the equation \eqref{ode-3} has two global solutions $u_\alpha$ and $v_\alpha$ such that
\begin{align*}
u_\alpha(x) &=\left(\zeta'\right)^{-1/2}\left(\Ai(\alpha^{2/3}\zeta)+\epsilon_\alpha(x)\right),\\
v_\alpha(x) &=\left(\zeta'\right)^{-1/2}\left(\Bi(\alpha^{2/3}\zeta)+\eta_\alpha(x)\right),
\end{align*}
where
\begin{align*}
|\epsilon_\alpha(x)| &\leq \frac{1}{\lambda}\frac{M(\alpha^{2/3}\zeta)}{E(\alpha^{2/3}\zeta)}\left(e^{\lambda\alpha^{-1}J}-1\right),\\
|\eta_\alpha(x)| &\leq \frac{1}{\lambda}M(\alpha^{2/3}\zeta)E(\alpha^{2/3}\zeta)\left(e^{\lambda\alpha^{-1}J}-1\right).
\end{align*}
\end{thm}

Here $\lambda$ is a positive universal constant whose exact value will play no role for us. Notice that the bounds on $\epsilon$ and $\eta$ in \cite{olverbook} are expressed in terms of a more precise \emph{error-control function} which we trivially bound by the constant $J$.

Summarizing, we associated to our potential $U\in\Pot_+^\pdeg(\co)$ the following objects: the turning (or transition) point $x_0$, the new variable $\zeta$, the quantities $b$ and $J$, and the family of solutions $\left(u_\alpha\right)_{\alpha>0}$ and $(v_\alpha)_{\alpha>0}$, which will be later shown to be respectively recessive and dominant (that is, square-integrable and not square-integrable).

\subsection{A useful lemma} We discuss here a technical lemma that will be needed later. First of all, fix $0<y_0<y_1$. If $\gamma>-1$, and $f:(0,y_1)\rightarrow \RR$ is a continuous function, we define 
\[
g_{y_0,y_1}^\gamma(f)(x) \defeq
\begin{cases}
(y_0-x)^{-\gamma-1}\int_x^{y_0}(y_0-y)^\gamma f(y) \,dy, & x\in (0,y_0),\\
(x-y_0)^{-\gamma-1}\int_{y_0}^x(y-y_0)^\gamma f(y) \,dy, & x\in (y_0,y_1).
\end{cases}
\] 

Next, given $\rho\leq0$ and $K\in\Rpos$, we denote by $\mathcal{C}^\rho_{y_1}(K)$ the collection of $C^2$ functions $f:(0,y_1)\rightarrow\Rpos$ such that
\[
K^{-1}\leq f(x)\leq K, \quad |f'(x)|\leq K, \quad |f''(x)|\leq Kx^\rho\qquad\forall x \in(0,y_1).
\]
If $\rho=0$, the last inequality just means that $f''$ is bounded on the whole interval, while if $\rho<0$ it is allowed to be singular at $0$.

\begin{lem}\label{tech-lem} Let $y_0, y_1, K, \rho$ be as above, and $f,g\in \mathcal{C}^\rho_{y_1}(K)$.
\begin{enumerate}[label=(\roman*)]
\item\label{en:tech1} There exists $\widetilde{K}=\widetilde{K}(K,\rho,y_1)$ such that $f\cdot g\in \mathcal{C}^\rho_{y_1}(\widetilde{K})$.
\item\label{en:tech2} For every $r\in\RR$ there exists $\widetilde{K}=\widetilde{K}(K,\rho,y_1,r)$ such that $f^r\in \mathcal{C}^\rho_{y_1}(\widetilde{K})$.
\item\label{en:tech3} For every $\gamma>-1$ there exists $\widetilde{K}=\widetilde{K}(K,\rho,y_0,y_1,\gamma)$ such that
\[
g^\gamma_{y_0,y_1}(f)\in \mathcal{C}^\rho_{y_1}(\widetilde{K}).
\]
\end{enumerate}
\end{lem}
A couple of comments may be useful.
\begin{enumerate}
\item The functions in the class $\mathcal{C}^\rho_{y_1}(\widetilde{K})$ are positive, thus their real powers appearing in part \ref{en:tech2} are well-defined.
\item Part \ref{en:tech3} above implicitly states that $g^\gamma_{y_0,y_1}(f)$, initially defined for $x\neq y_0$, has a $C^2$ extension to the whole interval.  
\end{enumerate}

\begin{proof}
The proofs of parts \ref{en:tech1} and \ref{en:tech2} are pretty straightforward, so we limit ourselves to check the bound on second derivatives of $f\cdot g$, leaving the other computations to the reader. If $f,g\in \mathcal{C}^\rho_{y_1}(K)$ we have
\[\begin{split}
|(f\cdot g)''(x)| &= |2f'(x)g'(x)+f(x)g''(x)+f''(x)g(x)|\\
&\leq 2K^2 + 2K^2x^\rho\leq 2K^2y_1^{-\rho}x^\rho + 2K^2x^\rho,
\end{split}\]
and hence $\widetilde{K} \defeq 2K^2(y_1^{-\rho}+1)$ works.

The proof of part \ref{en:tech3} requires just a bit more care. Observing that we have
\[
(x-y_0)^{\gamma+1}=(\gamma+1)\int_{y_0}^x(y-y_0)^\gamma \,dy\qquad (x>y_0),
\]
and the analogous formula for $x<y_0$, it is clear that $g^\gamma_{y_0,y_1}(f)$ has a continuous extension to $x_0$. Differentiating the integration by parts formula:
\[
g^\gamma_{y_0,y_1}(f)(x)=\frac{f(x)}{\gamma+1}-(x-y_0)^{-\gamma-1}\int_{y_0}^x\frac{(y-y_0)^{\gamma+1}}{\gamma+1}f'(y) \,dy\qquad (x>y_0)
\]
and the corresponding one for $x<y_0$, we find
\[
\frac{d}{dx}g^\gamma_{y_0,y_1}(f)(x)=g^{\gamma+1}_{y_0,y_1}(f')(x)\qquad\forall x\neq y_0.
\]
In particular, from what we said above, $g^\gamma_{y_0,y_1}(f)$ is $C^1$. Iterating the procedure, we see that it is in fact $C^2$ and
\[
\frac{d^2}{dx^2}g^\gamma_{y_0,y_1}(f)(x)=g^{\gamma+2}_{y_0,y_1}(f'')(x).
\]
Moreover, if $x<y_0$, then
\[\begin{split}
\left|\frac{d^2}{dx^2}g^\gamma_{y_0,y_1}(f)(x)\right| 
&\leq (y_0-x)^{-\gamma-3}\int_x^{y_0}(y_0-y)^{\gamma+2} |f''(y)| \,dy\\
&\leq Kx^\rho(y_0-x)^{-\gamma-3}\int_x^{y_0}(y_0-y)^{\gamma+2} \,dy\leq \frac{K}{\gamma+3}x^\rho, 
\end{split}\]
while, if $x>y_0$, then
\[\begin{split}
\left|\frac{d^2}{dx^2}g^\gamma_{y_0,y_1}(f)(x)\right|
&\leq (x-y_0)^{-\gamma-3}\int_{y_0}^x(y-y_0)^{\gamma+2} |f''(y)| \,dy\\
&\leq \frac{Ky_0^\rho}{\gamma+3}\leq \frac{Ky_0^\rho}{(\gamma+3)y_1^\rho}x^\rho.
\end{split}\]
We omit the easier bounds on $g^\gamma_{y_0,y_1}(f)$ and its first derivative.
\end{proof}

\subsection{Bounding $J$}\label{J-sec}

We can now discuss the main step in the proof of Proposition \ref{trans-prp}, that is, Lemma \ref{J-lem} below. This relies on a few facts that we proceed to state and prove.

\begin{prp}\label{x1-prp}
We have $x_0\simeq1$ and $b\simeq1$. Moreover, there exists $x_1>x _0$ such that $x_1\simeq 1$ and the following estimates hold for $x\geq x_1$:
\[
\sqrt{U(x)-1}\simeq x^{\pdeg/2}, \qquad \zeta(x)\simeq x^{(\pdeg+2)/3}, \qquad \zeta'(x)\simeq \zeta(x)^{(\pdeg-1)/(\pdeg+2)}.
\]
\end{prp}
\begin{proof}
We are going to use Definition \ref{pot-dfn} many times without comment.

The estimate on the transition point is obvious: $x_0 \simeq U(x_0)^{1/\pdeg}=1$. Then
\[
b = \left(\frac{3}{2}\int_0^{x_0}\sqrt{1-U}\right)^{2/3}\lesssim x_0^{2/3}\simeq 1.
\]
If we define $x_0'$ by $U(x_0')=1/2$, we also have $x_0'\simeq1$ and
\[
b = \left(\frac{3}{2}\int_0^{x_0}\sqrt{1-U}\right)^{2/3} \geq \left(\frac{3}{2}\int_0^{x_0'}\frac{1}{\sqrt{2}}\right)^{2/3}\simeq 1.
\]
Defining $x_1'$ by $U(x_1')=2$, we have $x_1'\simeq1$ and, for every $x\geq x_1'$,
\[
x^{\pdeg/2}\lesssim \sqrt{\frac{U(x)}{2}}\leq \sqrt{U(x)-1}\leq \sqrt{U(x)}\lesssim x^{\pdeg/2},
\]
and
\[\begin{split}
\zeta(x)^{3/2} 
&\simeq \int_{x_0}^x\sqrt{U(x)-1}\\
&\simeq \int_{x_0}^{x_1'}\sqrt{U(x)-1} +\int_{x_1'}^x x^{\pdeg/2}\\
&\simeq \int_{x_0}^{x_1'}\sqrt{U(x)-1}+ x^{1+\pdeg/2}-(x_1')^{1+\pdeg/2}
\end{split}\]
Since $\int_{x_0}^{x_1'}\sqrt{U(x)-1}\leq x_1'\lesssim1$, it is clear that we can choose $x_1\simeq1$ such that $\zeta(x)^{3/2}\simeq x^{1+\pdeg/2}$ when $x\geq x_1$, as we wanted. 

Finally, differentiating the identity $\frac{2}{3}\zeta(x)^{3/2}=\int_{x_0}^x\sqrt{U-1}$ for $x>x_0$ and the corresponding identity for $x<x_0$, we find
\begin{equation}\label{eq:der_ch_var}
\zeta'=\sqrt{\frac{U-1}{\zeta}} \qquad (x\neq x_0).
\end{equation}
Notice that $U>1$ when $\zeta>0$ and $U<1$ when $\zeta<0$. The estimate of $\zeta'$ on $[x_1,+\infty)$ follows from what we just proved.
\end{proof}

Consider next the auxiliary function
\[
\beta(x) \defeq \begin{cases}
\frac{U(x)-1}{x-x_0}, &x\neq x_0,\\
U'(x_0), &x=x_0,
\end{cases}
\]
which is clearly positive and continuous on $\Rpos$ and $C^3$ on $\Rpos\setminus\{x_0\}$.

\begin{prp}\label{id-prp} The following identity holds:
\[
\zeta'=\left(\frac{2}{3}\right)^{1/3} \sqrt{\beta}\cdot [g^{1/2}_{x_0,x_1}(\sqrt{\beta})(x)]^{-1/3}\qquad\forall x \in (0,x_1) \setminus \{x_0\}.
\]
\end{prp}

\begin{proof}
From the identity \eqref{eq:der_ch_var} it follows that
\[
\zeta'=\sqrt{\frac{U-1}{\zeta}}=\sqrt{\beta}\left(\sqrt{\frac{\zeta}{x-x_0}}\right)^{-1}\qquad (x\neq x_0).
\] 
If $x>x_0$, then
\[\begin{split}
\sqrt{\frac{\zeta}{x-x_0}} 
&= \left(\frac{3}{2}\right)^{1/3} \left(|x-x_0|^{-{3/2}}\int_{x_0}^x\sqrt{U-1}\right)^{1/3}\\
&=\left(\frac{3}{2}\right)^{1/3} \left(|x-x_0|^{-{3/2}}\int_{x_0}^x (y-x_0)^{1/2}\sqrt{\frac{U(y)-1}{y-x_0}} \,dy\right)^{1/3}\\
&=\left(\frac{3}{2}\right)^{1/3} g^{1/2}_{x_0,x_1}(\sqrt{\beta})(x)^{1/3}.
\end{split}\]
Analogously one can see that the same identity holds also for $x<x_0$. 
\end{proof}

\begin{prp}\label{beta-prp}
$\beta\in\mathcal{C}^{\rho}_{x_1}(K_1)$, where $K_1\lesssim 1$ and $\rho=\min\{\pdeg-2,0\}$.
\end{prp}
Here $x_1$ is the constant appearing in Proposition \ref{x1-prp}.

\begin{proof}
\emph{Estimate of $\beta(x)$:} Let $x_0'$ and $x_0''$ be such that $U(x_0')=1/2$ and $U(x_0'')=3/2$, so that
\[
\frac{1}{2}=U(x_0)-U(x_0')=U'(\bar{x})(x_0-x_0') \simeq\bar{x}^{\pdeg-1}(x_0-x_0'), 
\]
for some $\bar{x}\in[x_0',x_0]$. Since $x_0', x_0\simeq1$, this implies $(x_0-x_0')\simeq 1$. Analogously, one shows that $(x_0''-x_0)\simeq 1$. Therefore, when $x\in (0,x_1]\setminus (x_0',x_0'')$, we have $\beta(x)\simeq |U(x)-1|$ and, using the monotonicity of $U$, we conclude that
\[
\frac{1}{2}=\min\{1-U(x_0'), U(x_0'')-1\}\leq |U(x)-1|\lesssim 1+x_1^{\pdeg}\lesssim 1.
\]
If instead $x\in (x_0',x_0'')\setminus\{x_0\}$, then
\[
\beta(x)=\frac{U(x)-U(x_0)}{x-x_0}=U'(\bar{x})\simeq \bar{x}^{\pdeg-1},
\]
for some $\bar{x}$ between $x$ and $x_0$. Hence $\beta(x)\simeq 1$ on $(0,x_1]$.

\medskip\emph{Estimate of $|\beta'(x)|$:} We have
\[
\beta'(x)=\frac{U'(x)(x-x_0)-U(x)+1}{(x-x_0)^2}\qquad (x\neq x_0).
\]
If $x\in  (0,x_1]\setminus (x_0',x_0'')$, then $|\beta'(x)|\lesssim U'(x)+U(x)+1\lesssim x^{\pdeg-1} + x^{\pdeg}+1\lesssim 1$ (notice that $\pdeg>1$ and $x^{\pdeg-1}$ is bounded at $0$). If instead $x\in (x_0',x_0'')$, then we look at the expansion of $U$ up to second order:
\[
1=U(x_0)=U(x) + U'(x)(x_0-x) + \frac{U''(\bar{x})}{2}(x_0-x)^2,
\]
for some $\bar{x}$ between $x$ and $x_0$. This allows us to write $\beta'(x) = U''(\bar{x})/2$, which implies easily that $\beta$ is $C^1$ and $|\beta'(x)|\lesssim1$ on $(0,x_1]$.

\medskip\emph{Estimate of $|\beta''(x)|$:} We have
\[
\beta''(x)=\frac{U''(x)(x-x_0)^2-2U'(x)(x-x_0)+2U(x)-2}{(x-x_0)^3}\qquad (x\neq x_0).
\]
If $x\in (0,x_1] \setminus (x_0',x_0'')$, then
\[\begin{split}
|\beta''(x)|
&\lesssim |U''(x)|+U'(x)+U(x)+1\\
&\lesssim x^{\pdeg-2}+x^{\pdeg-1} + x^{\pdeg}+1\lesssim x^{\rho},
\end{split}\]
where $\rho \defeq \min\{\pdeg-2,0\}$. If instead $x\in (x_0',x_0'')$, then we look at the expansion of $U$ up to third order:
\[
1=U(x_0)=U(x) + U'(x)(x_0-x) + \frac{U''(x)}{2}(x_0-x)^2+\frac{U'''(\bar{x})}{6}(x_0-x)^3,
\]
for some $\bar{x}$ between $x$ and $x_0$. Analogously as above, this allows us to write $\beta''(x) = U'''(\bar{x})/3$, and conclude that $\beta$ is $C^2$ and $\beta''(x)\lesssim x^{\rho}$ on $(0,x_1]$.
\end{proof}

Combining Proposition \ref{id-prp}, Proposition \ref{beta-prp}, and Lemma \ref{tech-lem} immediately yields the following result.

\begin{prp}\label{zeta-prp}
$\zeta'\in\mathcal{C}^\rho_{x_1}(K_2)$ with $K_2\lesssim 1$ and $\rho=\min\{\pdeg-2,0\}$. In particular $\zeta$ is a $C^3$ diffeomorphism of $\Rpos$ onto $(-b,+\infty)$. 
\end{prp}

We can finally prove the fundamental uniform bound on the quantity $J$ defined in \eqref{eq:J}, that will allow us to obtain uniform estimates from Olver's result.

\begin{lem}\label{J-lem}
$J\lesssim1$.
\end{lem}

\begin{proof}
By Proposition \ref{zeta-prp}, the Schwarzian derivative $\Phi(\zeta)$ defined in \eqref{schwarzian} is a well-defined continuous function of $\zeta\in(-b,+\infty)$.  We split the integral in \eqref{eq:J} as follows:
\[
J=\left(\int_{-b}^{-b/2} + \int_{-b/2}^{\zeta_1} + \int_{\zeta_1}^{+\infty}\right)\frac{|\Phi(\zeta)|}{\sqrt{|\zeta|}} \,d\zeta \eqdef J_1+J_2+J_3,
\]
where $\zeta_1 \defeq \zeta(x_1)$.

We start by rewriting formula \eqref{schwarzian} as
\[
\Phi(\zeta(x)) = \zeta'(x)^{-1/2} \frac{d^2}{d\zeta^2} \left[\zeta'(x)^{1/2} \right]
= \zeta'(x)^{-3/2} \frac{d}{dx} \left[ \zeta'(x)^{-1} \frac{d}{dx} \left[ \zeta'(x)^{1/2}\right] \right].
\]
By Proposition \ref{zeta-prp} and Lemma \ref{tech-lem}, we have that $(\zeta')^{1/2}, (\zeta')^{-1}, (\zeta')^{-3/2}\in\mathcal{C}^\rho_{x_1}(K_3)$ with $K_3\lesssim 1$, and hence
\begin{equation}\label{phi-bound}
|\Phi(\zeta(x))|\leq K_3^3+K_3^3 \,x^\rho\lesssim x^\rho\qquad\forall x\leq x_1.
\end{equation}
Now denote by $x_{-1}>0$ the point such that $\zeta(x_{-1})=-b/2$. Then
\[
1\simeq\frac{b}{2}=\zeta(x_{-1})-(-b)=\int_0^{x_{-1}}\zeta'(x) \,dx\simeq x_{-1},
\]
where we used Proposition \ref{zeta-prp} and the fact that $x_{-1}<x_0<x_1$. Thus \eqref{phi-bound} gives
\begin{equation}\label{phi-bound-2}
|\Phi(\zeta(x))|\lesssim (x_{-1})^\rho \lesssim 1\qquad\forall x\in (x_{-1},x_1).
\end{equation}

We can now take care of the first two parts of the $J$ integral. We have
\[\begin{split}
J_1 
&= \int_{-b}^{-b/2}\frac{|\Phi(\zeta)|}{\sqrt{|\zeta|}} \,d\zeta \leq \sqrt{\frac{2}{b}}\int_{-b}^{-b/2}|\Phi(\zeta)| \,d\zeta\\
&=\sqrt{\frac{2}{b}}\int_0^{x_{-1}}|\Phi(\zeta(x))| \,\zeta'(x) \,dx \\
&\lesssim \int_0^{x_{-1}}x^\rho \,dx \lesssim 1,
\end{split}\]
where we used \eqref{phi-bound}, $b \simeq 1 \simeq x_{-1}$, $\zeta'\simeq1$ on $(0,x_1)$ (by Proposition \ref{zeta-prp}), and $\rho=\min\{\pdeg-2,0\}>-1$. We also have 
\[
J_2 = \int_{-b/2}^{\zeta_1}\frac{|\Phi(\zeta)|}{\sqrt{|\zeta|}}d\zeta \lesssim \int_{-b/2}^{\zeta_1}\frac{d\zeta}{\sqrt{|\zeta|}}\lesssim1,
\]
where we used \eqref{phi-bound-2} and $b \simeq 1 \simeq \zeta_1$.

To take care of the region where $\zeta\geq\zeta_1$, we recall (see \cite[Chapter 11, eq.\ (3.06)]{olverbook}) that the definition of $\Phi$ and a straightforward computation give
\[
\Phi(\zeta(x)) = \zeta(x) \cdot \left(\frac{4f(x) f''(x)-5f'(x)^2}{16f(x)^3}\right) + \frac{5}{16\zeta(x)^2}, 
\]
where $f(x)=U(x)-1$. By Definition \ref{pot-dfn} and Proposition \ref{x1-prp},
\[
\left|\frac{4f(x)f''(x)-5(f'(x))^2}{16f(x)^3}\right|\lesssim x^{-\pdeg-2}\simeq \zeta(x)^{-3}\qquad\forall x\geq x_1.
\]
Thus $|\Phi(\zeta)|\lesssim \zeta^{-2}$ when $\zeta\geq\zeta_1$ and
\[
J_3=\int_{\zeta_1}^{+\infty}\frac{|\Phi(\zeta)|}{\sqrt{|\zeta|}}d\zeta\lesssim\int_{\zeta_1}^{+\infty}\zeta^{-2}d\zeta\lesssim 1.
\]
This completes the proof.
\end{proof}

\subsection{$L^2$ norms of Olver's solutions}

Lemma \ref{J-lem} proves that the hypotheses of Theorem \ref{olver-thm} are satisfied, so we are allowed to consider the one-parameter families of Olver solutions $\{u_\alpha\}_{\alpha>0}$ and $\{v_\alpha\}_{\alpha>0}$. The next two propositions state that they are respectively recessive and dominant, and do so in a quantitative fashion.

\begin{prp}\label{notL2-prp} If $\alpha>\frac{\lambda}{\log(\frac{\lambda}{\sqrt{2}}+1)} J$, then $v_\alpha\notin L^2(\Rpos)$.
\end{prp}
Recall that $\lambda$ is the universal positive constant appearing in Theorem \ref{olver-thm}. What matters for us is that, combining Proposition \ref{notL2-prp} and Lemma \ref{J-lem}, we find a threshold for $\alpha$ which is uniform in our class of potentials.

\begin{proof}
Recalling the definitions of $\eta_\alpha$, $c$, and Olver's auxiliary functions (see Section \ref{olver-sec}), we have
\[
|\eta_\alpha(x)|\leq \frac{\sqrt{2}}{\lambda}\Bi(\alpha^{2/3}\zeta(x))\left(e^{\lambda\alpha^{-1}J}-1\right)=(1-\tau) \Bi(\alpha^{2/3}\zeta(x)),\qquad \forall x\colon  \alpha^{2/3}\zeta(x)\geq c,
\]
where $\tau$ is positive, due to the assumption on $\alpha$. Thus
\[\begin{split}
\int_0^{+\infty}v_\alpha^2(x) \,dx
&\geq \tau \int_{\alpha^{2/3}\zeta(x)\geq c}\Bi(\alpha^{2/3}\zeta(x))^2  \,\zeta'(x)^{-1} \,dx\\
&= \tau\int_{\alpha^{2/3}\zeta\geq c}\Bi(\alpha^{2/3}\zeta)^2 \,(\zeta')^{-2} \,d\zeta,
\end{split}\]
where we changed variable of integration, and of course $\zeta' = \zeta'(x(\zeta))$. 
Since $c<0$ and $\zeta_1=\zeta(x_1)>0$, we can restrict the interval of integration to $(\zeta_1,+\infty)$ and use Proposition \ref{x1-prp} to bound this integral from below by
\[
\tau \int_{\zeta_1}^{+\infty}\Bi(\alpha^{2/3}\zeta)^2 \,\zeta^{-(2\pdeg-2)/(\pdeg+2)} \,d\zeta.
\]
This integral diverges because $\Bi(t) \sim \pi^{-1/2}t^{-1/4}\exp(\frac{2}{3}t^{3/2})$ as $t\rightarrow+\infty$.
\end{proof}

\begin{prp}\label{L2-prp} There exists $\alpha_0\simeq1$ such that, for every $\alpha\geq\alpha_0$,
\begin{equation}\label{u-L2}
\int_{\Rpos} u_\alpha(x)^2 \,dx \simeq \alpha^{-1/3},
\end{equation}
and 
\begin{equation}\label{Uu-L2}
\int_{\Rpos} U(x) \,u_\alpha(x)^2 \,dx \gtrsim \alpha^{-1/3}.
\end{equation}
\end{prp}
The second inequality above will be used in the proof of the virial-type bound in Section \ref{s:schroedinger}.

\begin{proof}
We change variable as above and split the integral as follows:
\[\begin{split}
\int_{\Rpos}u_\alpha(x)^2 \,dx 
&= \left(\int_{-b}^{c\alpha^{-2/3}}+\int_{c\alpha^{-2/3}}^{\zeta_1}+\int_{\zeta_1}^{+\infty}\right)u_\alpha(x(\zeta))^2 \, (\zeta')^{-1} \,d\zeta\\
&\eqdef I_{\zosc} + I_{\ztrans} + I_{\zexp},
\end{split}\]
where $\zeta_1=\zeta(x_1)$ as before, and the subscripts stand for \emph{oscillatory}, \emph{transition}, and \emph{exponential(ly decaying)}. 

In the \emph{oscillatory} region $\{\zeta\leq c\alpha^{-2/3}\}$: 
\[
\epsilon_\alpha(x(\zeta)) \leq \delta(\alpha)\sqrt{\Ai(\alpha^{2/3}\zeta)^2 +\Bi(\alpha^{2/3}\zeta)^2}, 
\]
with $\delta(\alpha) \defeq \lambda^{-1}\left(e^{\lambda\alpha^{-1}J}-1\right)$. Recall the asymptotics at $-\infty$ of the Airy functions \cite[pp.\ 392--393]{olverbook}:
\begin{align*}
\Ai(t) &= \frac{1}{\sqrt{\pi}|t|^{1/4}}\cos\left(\frac{2}{3}|t|^{3/2}-\frac{\pi}{4}\right) + \BigO(|t|^{-7/4}),\\
\Bi(t) &= -\frac{1}{\sqrt{\pi}|t|^{1/4}}\sin\left(\frac{2}{3}|t|^{3/2}-\frac{\pi}{4}\right) + \BigO(|t|^{-7/4}).
\end{align*}
If $c'<2c$, then it is easy to deduce that
\[
C^{-1}|c'|^{1/2}\leq\int_{c'}^c \Ai(t)^2 \,dt, \, \int_{c'}^c \Bi(t)^2 \,dt \leq C|c'|^{1/2},
\]
where the constant $C$ is universal.
Therefore, if $b\alpha^{2/3}>2|c|$, then
\[\begin{split}
\int_{-b}^{c\alpha^{-2/3}} \epsilon_\alpha(x(\zeta))^2 \,d\zeta
&\leq \delta(\alpha)^2\alpha^{-2/3}\left(\int_{-b\alpha^{2/3}}^{c}\Ai(t)^2 \,dt +\int_{-b\alpha^{2/3}}^{c}\Bi(t)^2 \,dt\right)\\
&\leq 2C\delta(\alpha)^2b^{1/2}\alpha^{-1/3},
\end{split}\]
and
\[
\int_{-b}^{c\alpha^{-2/3}} \Ai(\alpha^{2/3}\zeta)^2 \,d\zeta = \alpha^{-{2/3}}\int_{-b\alpha^{2/3}}^{c}\Ai(t)^2 \,dt\simeq\alpha^{-{1/3}}.
\]
Since $b, J\simeq1$ and $\zeta'\simeq1$ in the oscillatory region (Proposition \ref{zeta-prp}), there exists $\alpha_0\simeq1$ such that for every $\alpha\geq\alpha_0$ we have
\[
I_{\zosc} \simeq\int_{-b}^{c\alpha^{-2/3}}(\Ai(\alpha^{2/3}\zeta)+\epsilon_\alpha(x(\zeta)))^2 \,d\zeta \simeq\alpha^{-1/3}.
\]

In the \emph{transition} region $\{c\alpha^{2/3}\leq\zeta\leq \zeta_1\}$ we have $\zeta'\simeq1$ by Proposition \ref{zeta-prp}, and
\[
|\epsilon_\alpha(x(\zeta))| \leq\sqrt{2} \, \delta(\alpha) \,\Ai(\alpha^{2/3}\zeta)\lesssim \Ai(\alpha^{2/3}\zeta)\qquad \forall \alpha\geq\alpha_0.
\]
Notice that $\alpha_0\simeq1$ has been fixed above, and hence the bound has the uniformity we want. Therefore
\[\begin{split}
I_{\ztrans}
&\simeq \int_{c\alpha^{-2/3}}^{\zeta_1}(\Ai(\alpha^{2/3}\zeta)+\epsilon_\alpha(x(\zeta))^2 \,d\zeta\\
&\lesssim \alpha^{-2/3}\int_{c}^{\zeta_1\alpha^{2/3}}\Ai(t)^2 \,dt\lesssim \alpha^{-2/3},
\end{split}\]
because $\Ai$ is square-integrable on $\Rpos$.

Finally, observe that in the \emph{exponential} region $\{\zeta\geq\zeta_1\}$ Proposition \ref{x1-prp} gives $\zeta'\simeq\zeta^{(\pdeg-1)/(\pdeg+2)}$, while $|\epsilon_\alpha(x(\zeta))| \lesssim \Ai(\alpha^{2/3}\zeta)\ \forall \alpha\geq\alpha_0$ as in the transition region. Thus
\[\begin{split}
I_{\zexp}
&= \int_{\zeta_1}^{+\infty}(\Ai(\alpha^{2/3}\zeta)+\epsilon_\alpha(x(\zeta))^2 \, (\zeta')^{-2} \,d\zeta\\
&\lesssim \int_{\zeta_1}^{+\infty}\Ai(\alpha^{2/3}\zeta)^2 \, \zeta^{-(2\pdeg-2)/(\pdeg+2)} \,d\zeta\\
&\leq \zeta_1^{-(2\pdeg-2)/(\pdeg+2)}\int_{\zeta_1}^{+\infty}\Ai(\alpha^{2/3}\zeta)^2 \,d\zeta\\
&\lesssim \alpha^{-2/3}\int_{\alpha^{2/3}\zeta_1}^{+\infty}\Ai(\zeta)^2d\zeta\lesssim \alpha^{-2/3}.
\end{split}\]
Putting the estimates for the three integrals together (and enlarging $\alpha_0$ while keeping it $\simeq1$, if needed), we obtain \eqref{u-L2}.

To prove estimate \eqref{Uu-L2} observe that
\[\begin{split}
\int_{\Rpos} U(x) \, u_\alpha(x)^2 \,dx 
&= \int_{-b}^{+\infty} U(x(\zeta))  \, u_\alpha(x(\zeta))^2 \, (\zeta')^{-1} \,d\zeta \\
&\geq \int_{-b/2}^{c\alpha^{-2/3}} U(x(\zeta)) \, u_\alpha(x(\zeta))^2 \, (\zeta')^{-1} \,d\zeta \\ 
&\gtrsim  (x_{-1})^{\pdeg}\int_{-b/2}^{c\alpha^{-2/3}} u_\alpha(x(\zeta))^2 \, (\zeta')^{-1} \,d\zeta,
\end{split}\]
where $\zeta(x_{-1})=-b/2$ as in the proof of Lemma \ref{J-lem}. Now the same argument as before proves that the integral above is $\simeq \alpha^{-1/3}$ for $\alpha\geq\alpha_0$, where $\alpha_0\simeq1$.
\end{proof}

\subsection{Proof of Proposition \ref{trans-prp}}

The space of solutions of \eqref{ode-3} is $2$-dimensional, and hence it is spanned by the functions $u_\alpha$ and $v_\alpha$. If $\alpha\geq\alpha_0$ as in the statement of Proposition \ref{notL2-prp}, then we must have $u=\lambda u_\alpha$ for some $\lambda\in\RR$. Integrating both sides and using Proposition \ref{L2-prp}, we immediately conclude that $\lambda^2\simeq\alpha^{1/3}\int_{\Rpos} u^2$. Hence we have the pointwise estimate
\[\begin{split}
\frac{u(x)^2}{\int_{\Rpos} u^2}
&\simeq  \alpha^{1/3} u_\alpha(x)^2\\
&\lesssim \frac{\alpha^{1/3}}{\zeta'}\left(\Ai(\alpha^{2/3}\zeta)^2+\frac{M(\alpha^{2/3}\zeta)^2}{E(\alpha^{2/3}\zeta)^2}\left(e^{\lambda\alpha^{-1}J}-1\right)^2\right)\\
&\lesssim \frac{\alpha^{1/3}}{\zeta'}\left(\Ai(\alpha^{2/3}\zeta)^2+\frac{M(\alpha^{2/3}\zeta)^2}{E(\alpha^{2/3}\zeta)^2}\right),
\end{split}\]
where the second line follows from the discussion in Section \ref{olver-sec}, and the third one from $\alpha\geq\alpha_0\simeq1$ and Lemma \ref{J-lem}. Observe that we have the trivial bounds $\Ai(x)^2\leq M(x)^2 / E(x)^2$ and $E(x)^2\geq1$, and that
\[
M(x)\lesssim |x|^{-1/4}\qquad\forall x\in\RR.
\]
This inequality follows from the asymptotics at $+\infty$ and $-\infty$ of Airy functions (see, e.g., \cite[p.\ 395]{olverbook}). Putting everything together, we get
\begin{equation}\label{eq:olver_intermediate}
\frac{u(x)^2}{\int_{\Rpos} u^2}\lesssim \frac{\alpha^{1/3}}{\zeta'}M(\alpha^{2/3}\zeta)^2\lesssim \frac{\zeta^{-1/2}}{\zeta'} = \frac{1}{\sqrt{U-1}}.
\end{equation}
The last identity follows from \eqref{eq:der_ch_var}. Now, recalling the definition of $\beta$ in Section \ref{J-sec} and Proposition \ref{beta-prp}, we have
\[
|U(x)-1|=\beta(x)|x-x_0|\simeq |x-x_0|\qquad\forall x\leq x_1.
\]
If $x\geq x_1$, then Proposition \ref{x1-prp} gives $|U(x)-1|\simeq x^\pdeg\geq x_1^{\pdeg-1}|x-x_0|\gtrsim |x-x_0|$. In any case we obtain \eqref{eq:olver_trans}.

Note also that, from \eqref{eq:olver_intermediate} and the fact that $M$ is a bounded function, we obtain
\[
\frac{u(x)^2}{\int_{\Rpos} u^2}\lesssim \frac{\alpha^{1/3}}{\zeta'}.
\]
On the other hand, from Proposition \ref{zeta-prp} we know that $\zeta'(x) \simeq 1$ for $x \in (0,x_1)$, while, by Proposition \ref{x1-prp}, $\zeta'(x) \simeq x^{(\pdeg-1)/3} \gtrsim 1$ for $x \geq x_1$, and in any case \eqref{eq:olver_unif} follows.

Inequality \eqref{eq:olver_int} follows from Proposition \ref{L2-prp} arguing as above.

\section{Schr\"odinger operators}\label{s:schroedinger}

The goal of this section is to prove several precise estimates for eigenvalues and eigenfunctions of one-dimensional Schr\"odinger operators of the form
\[
-\partial_x^2+V,
\]
where the potential $V$ is in a suitable class of perturbations of power laws;  as in Section \ref{s:olver}, the key issue is obtaining \emph{uniform} estimates for all the potentials in a given class. Before delving into our results, let us recall some basic facts and establish notation.

We find it convenient to introduce the set $\GPot$ of potentials $V:\RR\rightarrow\Rnon$ satisfying the following conditions:\begin{enumerate}
\item $V(0)=0$,
\item $V$ is continuous,
\item $V$ is strictly increasing in $\Rpos$ and strictly decreasing in $\RR^{-}$,
\item $\lim_{|x|\rightarrow+\infty}V(x)=+\infty$.
\end{enumerate}

If $V\in \GPot$, then the operator $-\partial_x^2+ V$, defined on test functions, is essentially self-adjoint (see, e.g., \cite{reed-simon-II}, Theorem X.28).
We denote by $\opH^V$ its unique self-adjoint extension.
The spectrum of $\opH^V$ is well-known to consist of a sequence of simple positive eigenvalues tending to $+\infty$ (see, e.g., \cite{berezin-shubin}, Chapter 2). We denote by $E^V_n$ the $n$th eigenvalue (with respect to increasing order and starting from $n=1$) of $\opH^V$ and by $\psi^V_{n}$ a corresponding real-valued $L^2$-normalized eigenfunction. Basic results of Sturm--Liouville theory (cf.\ Appendix) tell us that $\psi_n^V(x)$ has a definite sign (positive or negative) for all $x>0$ sufficiently large; since $\psi^V_{n}$ is defined up to a sign, we choose $\psi_n^V$ so that $\psi_n^V(x)$ is positive for all sufficiently large $x>0$.
Finally, it will be useful to denote by $x_{n}^{V,+}$ (resp. $x_{n}^{V,-}$) the unique positive (resp. negative) solution of $V(x)=E^V_n$.

We recall a basic comparison result, which holds in far greater generality than in the class $\GPot$ (see, e.g., Section $4.5$ of \cite{davies}): if $V \leq c W$ pointwise for some $c \geq 1$, then
\[
\opH^V \leq c \opH^W
\]
in the sense of quadratic forms, whence
\begin{equation}\label{eq:comparison}
E_n^V \leq c E_n^W
\end{equation}
for all $n \in \Npos$.

We now define the classes of potentials we are actually interested in.

\begin{dfn}\label{pot2-dfn}
For $\co, \pdeg \in \Rpos$, we denote by $\Pot^\pdeg(\co)$ the collection of continuous functions $V : \RR \to \Rnon$ which are of class $C^3$ on $\RR \setminus \{0\}$ and satisfy
\begin{gather}
\co^{-1} |x|^{\pdeg} \leq V(x) \leq \co |x|^{\pdeg},\\
\co^{-1} |x|^{\pdeg-1} \leq |V'(x)| \leq \co|x|^{\pdeg-1},\\
|V''(x)| \leq \co |x|^{\pdeg-2}, \\
\label{eq:thirdderivative} |V'''(x)| \leq \co |x|^{\pdeg-3}
\end{gather}
for all $x \in \RR \setminus \{0\}$. Let $\Pot^\pdeg_{\even\convex}(\co)$ be the collection of even and convex functions in $\Pot^\pdeg(\co)$. Finally, let $\MPot^\pdeg(\co)$ be the collection of functions of the form $\scale V$ where $\scale \in \Rpos$ and $V \in \Pot^\pdeg(\co)$ (here $\mathcal{M}$ stands for ``multiples'').
\end{dfn}

Observe that $\MPot^\pdeg(\co)\subseteq\GPot$, so that the objects introduced above may be attached to any $V\in\MPot^\pdeg(\co)$.

As in the previous section, we fix $\pdeg,\co \in \Rpos$ with $\pdeg > 1$.
Before proceeding, let us highlight that, since our estimates will depend on $V$ only through the parameters $\pdeg$ and $\co$, in the following we will continue to use the symbols $\lesssim$ and $\simeq$ to conceal a constant depending only on them, and write $\simeq_k$, $\lesssim_k$ when the constant is allowed to depend on an additional parameter $k$. 

We now state the bounds that will be proved in the course of the section and will be crucial in the subsequent developments.

\begin{prp}[Riesz transform bound]\label{apriori_scaled-prp}
Let  $V \in \MPot^\pdeg(\co)$. Then, for all $k \in \NN$, the inequality
\[
\| V^k f \|_2  \lesssim_k \| (\opH^V)^k f \|_2
\]
holds for every $f$ in the domain of $(\opH^V)^k$.
\end{prp}

\begin{prp}[virial-type bound]\label{virial-prp}
For all $V \in \MPot^\pdeg(\co)$
and $n \in \Npos$,
\[
\int_\RR V(x) |\psi^V_{n}(x)|^2 \,dx \simeq E_n^V.
\]
\end{prp}

\begin{prp}[pointwise bounds]\label{eigenfunctions-prp}
There exist $\delta, \kappa \simeq 1$ such that, for all $V \in \MPot^\pdeg(\co)$
and $n \in \Npos$,
\[
|\psi^V_{n}(x)|^2 \lesssim \begin{cases}
|x_n^{V,\pm}|^{-2/3} \, (E_n^V)^{1/6} &\text{for all $x \in \RR$,}\\
|x_{n}^{V,\pm}|^{-1/2} \, |x-x^{V,\pm}_{n}|^{-1/2} &\text{if $\pm x >0$,}\\
(E_n^V)^{1/2} \exp(-\delta |x| \, V(x)^{1/2} ) &\text{if $\pm x \geq \kappa |x_{n}^{V,\pm}|$}
\end{cases}
\]
and moreover
\begin{equation}\label{eq:trans_eigen_lb}
(E_n^V)^{1/2} |x_n^{V,\pm}|  \gtrsim 1.
\end{equation}
\end{prp}

The notation $\pm$ in the latter proposition means that the two sets of inequalities obtained by replacing every occurrence of $\pm$ with either $+$ or $-$ hold true. 

\begin{rem}\label{rem:weaker-unif}
In Section \ref{s:grushin} we will only need the weaker $L^\infty$ bound
\[
|\psi^V_{n}(x)|^2 \lesssim |x_n^{V,\pm}|^{-1/2} \, (E_n^V)^{1/4},
\]
which may be derived from the one contained in Proposition \ref{eigenfunctions-prp} using \eqref{eq:trans_eigen_lb}.
\end{rem}

The next two remarks are important for the proofs of the propositions above.

\begin{rem}\label{rem:scaling}
Let $t \in \Rpos$, and let $T_t : L^2(\RR) \to L^2(\RR)$ be the isometric isomorphism defined by $T_t f(x) = \sqrt{t} f(tx)$. Assume that $V \in \MPot^\pdeg(\co)$ and define
\[
V_t = t^2 V(t \cdot).
\]
Then it is easily seen that $V_t \in \MPot^\pdeg(\co)$ as well (indeed, if $V = \scale W$ with $W \in \Pot^\pdeg(\co)$, then $t^{-\pdeg} W(t\cdot) \in \Pot^\pdeg(\co)$ and $V_t = \scale t^{2+\pdeg} t^{-\pdeg} W(t\cdot)$), and moreover
\[
\opH^{V_t} T_t = t^2 T_t \opH^V,
\]
whence, for all $n \in \Npos$,
\[
\psi_n^{V_t}= T_t\psi_n^V, \qquad
E_n^{V_t} = t^2 E_n^V, \qquad
x_n^{V_t,\pm} = t^{-1} x_n^{V,\pm}.
\]
From these formulas it is easily checked that the estimates in Propositions \ref{apriori_scaled-prp}, \ref{virial-prp} and \ref{eigenfunctions-prp} are invariant under the ``scaling'' $V \mapsto V_t$ (implicit constants included). Hence in proving those Propositions it will be enough to consider a suitably chosen scaled version $V_t$ in place of the original $V$. 
Indeed, note that:
\begin{itemize}
\item it is possible to choose $t$ so that $V_t \in \Pot^\pdeg(\co)$ (if $V = \scale W$ with $W \in \Pot^\pdeg(\co)$, then take $t = \scale^{-1/(\pdeg+2)}$);
\item for all $n \in \Npos$, there is also a choice of $t$ so that $(E_n^{V_t})^{-1} V_t \in \Pot^\pdeg(\co)$ (if $V = \scale W$ with $W \in \Pot^\pdeg(\co)$, then take $t = (E_n^V/\scale)^{1/\pdeg}$).
\end{itemize}
\end{rem}

\begin{rem}\label{rem:inverting}
A similar argument, exploiting the isometric isomorphism $S$ of $L^2(\RR)$ given by $Sf(x) = f(-x)$ shows that, for the potential $\check V = V(-\cdot)$, the following identities hold:
\[
|\psi_n^{\check V}(x)| = |\psi_n^{V}(-x)|, \qquad E_n^{\check V} = E_n^V, \qquad x_n^{\check V,\pm} = - x_n^{V,\mp}.
\]
Consequently it is enough to prove the pointwise estimates of Proposition \ref{eigenfunctions-prp} for $x \geq 0$.
\end{rem}

\subsection{Proof of the lower bound \eqref{eq:trans_eigen_lb}}

By Remark \ref{rem:scaling}, we may assume that $V \in \Pot^\pdeg(\co)$. Then $V(x) \geq \co^{-1} |x|^\pdeg$ and, by comparison \eqref{eq:comparison},
\[
E_n^V \geq \co^{-1} E_n^{|\cdot|^\pdeg} \gtrsim 1.
\]
Since $V(x) \leq \co|x|^\pdeg$, we also get
\[
|x_n^{V,\pm}| \gtrsim (E_n^V)^{1/\pdeg} \gtrsim 1.
\]
and the conclusion follows.

\subsection{Eigenfunction estimates: uniform bounds}\label{ss:eeunifbd}

Here we prove the first estimate of Proposition \ref{eigenfunctions-prp}. By Remark \ref{rem:scaling}, it is enough to consider the case where $V = E_n^V U$ for some $U \in \Pot^\pdeg(\co)$, hence the eigenfunction $\psi_n^V$ solves the ODE
\[
u'' = \alpha^2 (U - 1) u,
\]
where $\alpha = \sqrt{E_n^V}$. Note that, in this case,
$|x_n^{V,\pm}| \simeq 1$, 
so we are reduced to proving the estimate
\begin{equation}\label{eq:eigen_unif_rid}
\|\psi_n^V\|_\infty^2 \lesssim (E_n^V)^{1/6}.
\end{equation}

By Proposition \ref{trans-prp} (applied to $\psi_n^V(\pm \cdot)$), there exists $\alpha_0 \simeq 1$ so that \eqref{eq:eigen_unif_rid} holds whenever $E_n^V \geq \alpha_0^2$.

Suppose instead that $E_n^V \leq \alpha_0^2$.
Since $\psi_n^V$ is $L^2$-normalized, using the fundamental theorem of calculus we obtain
\[\begin{split}
\|\psi_n^V\|_\infty^2 &\leq 2 \|(\psi_n^V)'\|_2 \leq 2\left(\|(\psi_n^V)'\|_2^2 + \int_\RR V |\psi_n^V|^2\right)^{1/2}\\
&\leq 2 (E_n^V)^{1/2}\lesssim (E_n^V)^{1/6}.
\end{split}\]

\subsection{Eigenfunction estimates: exponential decay}
Exponential decay at $\infty$ of Schr\"odinger eigenfunctions is a well-understood phenomenon (also in several variables, see \cite{agmon}). For our purposes, it is enough to combine the uniform bound and the following one-dimensional estimate.

\begin{thm}
Assume that $V \in \GPot$. Then, for all $n \in \Npos$ and all $x' > x > \pm x_n^{V,\pm}$,
\[
|\psi_n^V(\pm x')| \leq |\psi_n^V(\pm x)| \, \exp\left(- \left|\int_{\pm x}^{\pm x'} \sqrt{V-E_n^V}\right|\right).
\]
\end{thm}
\begin{proof}
See \cite[Section 8.2]{titchmarsh}.
\end{proof}

By Remark \ref{rem:scaling} we may assume that $V \in \Pot^\pdeg(\co)$. Moreover, by Remark \ref{rem:inverting}, it is enough to prove the estimate for $x > 0$. From the uniform bound and the above estimate we deduce in particular that, if $x > x_n^{V,+}$,
\[
|\psi_n^V(x)|^2 \lesssim (E_n^V)^{1/2} (\sqrt{E_n^V} x_n^{V,+})^{-1/2} \exp\left( -\int_{x_n^{V,+}}^x \sqrt{V-E_n^V}\right).
\]
Since $V \in \Pot^\pdeg(\co)$, it is possible to find $\beta,\kappa \simeq 1$ so that $\kappa \geq 2$ and $V(t) - E_n^V \geq \beta^2 t^\pdeg$ for all $t \geq \kappa x_n^{V,+}$, whence, if $x\geq \kappa^2 x_n^{V,+}$,
\begin{multline*}
\int_{x_n^{V,+}}^x \sqrt{V-E_n^V} \geq \beta \int_{\kappa x_{n}^{V,+}}^x t^{\pdeg/2} \,dt 
= \beta(x^{\pdeg/2+1} - (\kappa x_n^{V,+})^{\pdeg/2+1}) \\
\geq \beta (1-(1/\kappa)^{\pdeg/2+1}) \, x^{\pdeg/2+1} \geq \beta (1-(1/\kappa)^{\pdeg/2+1}) \, \co^{-1/2} x V(x)^{1/2}.
\end{multline*}
Hence, if we set $\delta = \beta (1-(1/\kappa)^{\pdeg/2+1}) \co^{-1/2}$, we have $\delta\simeq1$ (since $\kappa\geq 2$) and, for all $x \geq \kappa^2 x_n^{V,+}$,
\[
|\psi_n^V(x)|^2 \lesssim (E_n^V)^{1/2} (\sqrt{E_n^V} x_n^{V,+})^{-1/2} \exp\left( -\delta x V(x)^{1/2} \right),
\]
and, since $\sqrt{E_n^V} x_n^{V,+} \gtrsim 1$ by \eqref{eq:trans_eigen_lb}, the conclusion follows by renaming $\kappa$.

\subsection{Eigenfunction estimates: bounds at the transition points}

By Remark \ref{rem:scaling}, it is enough to consider the case where $V = E_n^V U$ with $U \in \Pot^\pdeg(\co)$. As already discussed in \S\ref{ss:eeunifbd}, $\psi_n^V$ satisfies the ODE
\[
u'' = \alpha^2 (U-1) u
\]
with $\alpha =\sqrt{E_n^V}$, and moreover $|x_n^{V,\pm}| \simeq 1$, so the estimate to be proved reduces to
\begin{equation}\label{eq:ef_trans_bd_olver}
|\psi_n^V(x)|^2 \lesssim |x-x_n^{V,\pm}|^{-1/2}
\end{equation}
for all $\pm x > 0$.

By Proposition \ref{trans-prp} (applied to $\psi_n^V(\pm \cdot)$), there exists $\alpha_0 \simeq 1$ so that the required bound \eqref{eq:ef_trans_bd_olver} holds whenever $E_n^V \geq \alpha_0^2$.

Suppose now that $E_n^V \leq \alpha_0^2$. Note that, by \eqref{eq:trans_eigen_lb}, in this case $E_n^V \simeq 1$. Hence, when $\pm x \geq \kappa x_n^{V,\pm}$, the exponentially decaying bound gives us that
\[
|\psi_n^V(x)|^2 \lesssim \exp(-\delta |x| V(x)^{1/2}) \lesssim |x|^{-1/2} \leq |x-x_n^{V,\pm}|^{-1/2},
\]
while, if $\pm x \leq \kappa x_n^{V,\pm}$, then the uniform bound and $|x_n^{V,\pm}| \simeq 1$ imply that
\[
|\psi_n^V(x)|^2 \lesssim 1 \lesssim |x-x_n^{V,\pm}|^{-1/2}.
\]

\subsection{Proof of the virial-type estimate}
Note first that, since $\psi_n^V$ is an eigenfunction,
\[
\int_\RR V |\psi_n^V|^2  + \|(\psi_n^V)'\|_2^2 = \langle \opH_n^V \psi_n^V,\psi_n^V\rangle = E_n^V,
\]
hence the inequality $\int_\RR V |\psi_n^V|^2 \lesssim E_n^V$ is always true (with constant $1$) and it remains to prove the opposite estimate.

By Remark \ref{rem:scaling}, it is enough to consider the case where $V = E_n^V U$ for some $U \in \Pot^\pdeg(\co)$, so the inequality to be proved reduces to
\begin{equation}\label{eq:virial_olver_lb}
\int_\RR U |\psi_n^V|^2 \gtrsim 1.
\end{equation}
Once again, as in \S\ref{ss:eeunifbd}, the eigenfunction $\psi_n^V$ satisfies the ODE
\[
u'' = \alpha^2 (U-1) u
\]
with $\alpha = \sqrt{E_n^V}$, and moreover $|x_n^{V,\pm}| \simeq 1$. By inequality \eqref{eq:olver_int} of Proposition \ref{trans-prp} (applied to $\psi_n^V(\pm\cdot)$), there exists $\alpha_0 \simeq 1$ so that \eqref{eq:virial_olver_lb} holds whenever $E_n^V \geq \alpha_0^2$.

Suppose now that $E_n^V \leq \alpha_0^2$. Consider $U_{[t]} = t^{-\pdeg} U(t\cdot)$ and note that $U_{[t]} \in \Pot^\pdeg(\co)$ as well, for all $t\in\Rpos$. Hence
if $T_t : L^2(\RR) \to L^2(\RR)$ is defined as in Remark \ref{rem:scaling}, then
\[\begin{split}
1 \simeq E_1^{|\cdot|^\pdeg} \simeq E_1^{U_{[t]}} &\leq \int_\RR |(T_t \psi_n^V)'|^2 + \int_\RR U_{[t]} |T_t \psi_n^V|^2 \\
&=  t^2\int_\RR|(\psi_n^V)'|^2 + t^{-\pdeg}\int_\RR U |\psi_n^V|^2\\
&\leq t^2 E_n^V + t^{-\pdeg} \int_\RR U|\psi_n^V|^2\\
&\lesssim t^2 +t^{-\pdeg}\int_\RR U |\psi_n^V|^2
\end{split}\]
for all $t \in \Rpos$; in particular we can choose $t \simeq 1$ sufficiently small so that the term $t^2$ can be brought to the left-hand side, thus obtaining again the desired inequality \eqref{eq:virial_olver_lb}.

\subsection{Riesz transform bounds}
By Remark \ref{rem:scaling}, in order to prove Proposition \ref{apriori_scaled-prp}, it is enough to consider the case where $V \in \Pot^\pdeg(\co)$.

Let now $\mathcal{X}^V = \Span \{ \psi_n^V \tc n \in \Npos \}$ be the set of finite linear combinations of eigenfunctions of $\opH^V$. Clearly $\mathcal{X}^V$ is a core for each power of $\opH^V$. Moreover, in view of the exponential decay of eigenfunctions discussed above, $\mathcal{X}^V$ is made of rapidly decaying functions, hence it is contained in the domain of any power of $V$ (thought of as a multiplication operator on $L^2(\RR)$) and it is enough to prove the required estimates for all $f \in \mathcal{X}^V$. Therefore Proposition \ref{apriori_scaled-prp} follows from part \ref{en:apriori4} of the next result.

\begin{prp}\label{prp:apriori}
Let $V \in \Pot^\pdeg(\co)$ and $f \in \mathcal{X}^V$.
\begin{enumerate}[label=(\roman*)]
\item\label{en:apriori1} $\| f \|_2 \lesssim \| \opH^V f\|_2$;
\item\label{en:apriori2} $\| \partial_x f \|_2 \leq \|(\opH^V)^{1/2} f\|_2$ and $\|V^{1/2} f \|_2 \leq \|(\opH^V)^{1/2} f\|_2$;
\item\label{en:apriori3} $\| V f \|_2 \lesssim \| \opH^V f\|_2$, $\| \partial_x^2 f \|_2 \lesssim \| \opH^V f\|_2$, $\| V^{1/2} \partial_x f \|_2 \lesssim \| \opH^V f\|_2$.
\item\label{en:apriori4init} $\| V^2 f \|_2 \lesssim \| (\opH^V)^2 f\|_2$.
\item\label{en:apriori4} $\| V^k f \|_2 \lesssim_k \| (\opH^V)^k f\|_2$ for all $k \in \NN$.
\end{enumerate}
\end{prp}

The idea of the proof is to proceed inductively, starting from the elementary observation that $V \leq \opH^V$ in the sense of quadratic forms; the difficulty in passing to higher powers lies in the fact that $V$ and $\opH^V$ do not commute, hence one needs to control commutators, involving derivatives of $V$. The proof would be somehow simpler in the case $V(t) = |t|^\pdeg$ with $\pdeg \in \NN$, since iterated derivatives of $V$ would eventually vanish (this is used in \cite[Section 3]{gadzinski_semigroup_2000} in the case $\pdeg=1$); indeed, in the case $V(t) = |t|^\pdeg$ with $\pdeg \in 2\NN$, one could even deduce the result from known subelliptic estimates for a suitable homogeneous sub-Laplacian on a stratified Lie group, cf.\ \cite{nourrigat_inegalites_1987,shen_estimates_1995,robinson_grushin_2016}. Under our assumptions, however, we only have a finite order of differentiability of $V$, and derivatives of $V$ can diverge at $0$. For this reason some additional care is required in treating the first few steps of the induction scheme, which is reflected in the number of ``preliminary estimates'' listed in Proposition \ref{prp:apriori}. It should be noted that the proof does not require the estimate \eqref{eq:thirdderivative} of the third derivative of $V$ (or even its existence).

\begin{proof}
\ref{en:apriori1}. By comparison,
\[
E_1^{V} \gtrsim E_1^{|\cdot|^\pdeg} \gtrsim 1,
\]
hence
\[
\|\opH^V f\|_2 \geq E_1^V \|f\|_2 \gtrsim \|f\|_2.
\]

\ref{en:apriori2}. By the definition of $\opH^V$,	
\[
\|(\opH^V)^{1/2} f\|_2^2 = \langle \opH^V f,f \rangle = \langle -\partial_x^2 f,f\rangle + \langle V f,f \rangle = \|\partial_x f\|_2^2 + \|V^{1/2} f\|_2^2.
\]

\ref{en:apriori3}. Note that
\[
\|\opH^V f\|_2^2 = \langle \opH^V f, \opH^V f\rangle = \|\partial_x^2 f\|_2^2 + \|V f\|_2^2 - 2\Re \langle \partial_x^2 f, V f \rangle.
\]
Integration by parts gives that
\[
\langle \partial_x^2 f, V f \rangle = - \langle \partial_x f, (\partial_x V) f\rangle - \| V^{1/2} \partial_x f\|_2^2,
\]
hence
\begin{equation}\label{eq:norm_op_dec}
\|\opH^V f\|_2^2 -2\Re \langle \partial_x f, (\partial_x V) f\rangle = \|\partial_x^2 f\|_2^2 + \|V f\|_2^2 + 2\| V^{1/2} \partial_x f\|_2^2.
\end{equation}

If $1 < \pdeg \leq 2$, then $0 < \pdeg-1\leq \pdeg/2$, hence
\[
|\partial_x V| \lesssim 1 + V^{1/2}
\]
and
\[\begin{split}
|\langle \partial_x f, (\partial_x V) f\rangle| &\lesssim \| (\partial_x V) f \|_2^2 + \|\partial_x f\|_2^2 
\lesssim \|f\|_2^2 + \| V^{1/2} f \|_2^2+ \|\partial_x f\|_2^2 \\
&\lesssim \|f\|_2^2 + \|(\opH^V)^{1/2} f\|_2^2 \lesssim \|\opH^V f\|_2^2
\end{split}\]
by parts \ref{en:apriori1} and \ref{en:apriori2}. This, combined with \eqref{eq:norm_op_dec}, gives part \ref{en:apriori3}.

Suppose instead that $\pdeg > 2$. Then a further integration by parts gives that
\[
\langle \partial_x f, (\partial_x V) f\rangle = -\langle f, (\partial_x^2 V) f\rangle - \langle (\partial_x V)  f,  \partial_x f\rangle,
\]
whence
\[
2\Re \langle \partial_x f, (\partial_x V) f\rangle = -\langle f, (\partial_x^2 V) f\rangle
\]
and from \eqref{eq:norm_op_dec} we obtain that
\[
\|\opH^V f\|_2^2 + \langle f, (\partial_x^2 V) f\rangle = \|\partial_x^2 f\|_2^2 + \|V f\|_2^2 + 2\| V^{1/2} \partial_x f\|_2^2.
\]
On the other hand, $\pdeg-2 > 0$, because $\pdeg > 2$; so, by the assumptions on $V$,
\[
|\partial_x^2 V| \lesssim 1 + V,
\]
hence
\[
|\langle f, (\partial_x^2 V) f\rangle| \lesssim \|f\|_2^2 + \|V^{1/2} f\|_2^2 \leq \|f\|_2^2 + \|(\opH^V)^{1/2} f\|_2^2 \lesssim \|\opH^V f\|_2^2
\]
(here we have used parts \ref{en:apriori1} and \ref{en:apriori2}) and
\[
\|\partial_x^2 f\|_2^2 + \|V f\|_2^2 + 2\| V^{1/2} \partial_x f\|_2^2 \lesssim \|\opH^V f\|_2^2.
\]

\ref{en:apriori4init}. Note that
\[
(\opH^V)^2 f = -\partial_x^2 \opH^V f - V \partial_x^2 f + V^2 f,
\]
hence
\[
\|(\opH^V)^2 f\|_2^2 = \|\partial_x^2 \opH^V f + V \partial_x^2 f\|_2^2 + \|V^2 f\|_2^2 - 2\Re\langle \partial_x^2 \opH^V f, V^2 f\rangle - 2 \Re \langle V \partial_x^2 f , V^2 f\rangle.
\]
Now
\[
\langle V \partial_x^2 f, V^2 f\rangle 
= - \langle \partial_x f, \partial_x (V^3 f)\rangle 
= - \langle \partial_x f, (\partial_x (V^3)) f\rangle - \| V^{3/2} \partial_x f \|_2^2
\]
and moreover
\[
\langle \partial_x f, (\partial_x (V^3)) f\rangle = - \langle f, (\partial^2_x (V^3)) f\rangle - \langle f, (\partial_x (V^3)) \partial_x f\rangle,
\]
whence
\[
2 \Re \langle \partial_x f, (\partial_x (V^3)) f\rangle = - \langle f, (\partial^2_x (V^3)) f\rangle 
\]
and
\[
2 \Re \langle V \partial_x^2 f, V^2 f\rangle = \langle f, (\partial^2_x (V^3)) f\rangle - 2\| V^{3/2} \partial_x f\|_2^2.
\]
Therefore
\begin{multline*}
\|(\opH^V)^2 f\|_2^2 + 2\Re\langle \partial_x^2 \opH^V f, V^2 f\rangle + \langle f, (\partial_x^2 (V^3)) f\rangle \\
 = \|\partial_x^2 \opH^V f + V \partial_x^2 f\|_2^2 + \|V^2 f\|_2^2 + 2 \| V^{3/2} \partial_x f \|_2^2,
\end{multline*}
whence
\[
\|V^2 f\|_2^2 \leq \|(\opH^V)^2 f\|_2^2 + 2\Re\langle \partial_x^2 \opH^V f, V^2 f\rangle + \langle f, (\partial_x^2 (V^3)) f\rangle.
\]

Note now that
\[
\langle f, (\partial_x^2 (V^3)) f\rangle = 3 \langle V f, (2(\partial_x V)^2 + V (\partial_x^2 V) ) f \rangle, 
\]
and moreover, since $\pdeg > 1$,
\[
|\partial_x V|^2 + |V (\partial_x^2 V)| \lesssim 1+V^2.
\]
Hence
\[\begin{split}
|\langle f, (\partial_x^2 (V^3)) f\rangle| &\lesssim \|V f\|_2 (\|f\|_2 + \|V^2 f\|_2) \\
&\lesssim \|\opH^V f\|_2 (\| f\|_2 + \|V^2 f\|_2) \\
&\lesssim \|(\opH^V)^2 f\|_2^2 + \|(\opH^V)^2 f\|_2 \|V^2 f\|_2
\end{split}\]
by parts \ref{en:apriori1} and \ref{en:apriori3}. Moreover
\[
|\langle \partial_x^2 \opH^V f, V^2 f\rangle| \leq \|\partial_x^2 \opH^V f\|_2 \|V^2 f\|_2 \lesssim \|(\opH^V)^2 f\|_2 \|V^2 f\|_2
\]
by part \ref{en:apriori3}. Therefore
\[
\|V^2 f\|_2^2 \lesssim \|(\opH^V)^2 f\|_2^2 + \|(\opH^V)^2 f\|_2 \|V^2 f\|_2,
\]
from which the conclusion follows (if $\|V^2 f\|_2 \leq \|(\opH^V)^2 f\|_2$ then we are done, otherwise divide both sides of the previous inequality by $\|V^2 f\|_2$).

\ref{en:apriori4}. We prove the inequality
\[
\|V^k f\|_2 \lesssim_k \|(\opH^V)^k f\|_2
\]
by induction on $k \in \NN$. Note that the case $k=0$ is trivial and the cases $k=1,2$ have been treated in parts \ref{en:apriori3} and \ref{en:apriori4init}. Assume now that the inequality has been proved up to a certain $k \geq 2$ and let us prove it 
for $k+1$.

We first prove the auxiliary inequality
\begin{equation}\label{eq:induction_thesis2}
\|\opH^V (V^k f)\|_2 \lesssim_k \|(\opH^V)^{k+1} f\|_2.
\end{equation}
Note that
\[
\opH^V (V^k f) = V^k \opH^V f - 2 (\partial_x(V^k)) (\partial_x f) - (\partial_x^2(V^k)) f
\]
and, by the induction hypothesis,
\[
\|V^k \opH^V f\|_2 \lesssim_k \|(\opH^V)^{k+1} f\|.
\]
By the assumptions on $V$ (note that $k\geq 2$ and $\pdeg > 1$),
\[
|\partial_x^2 (V^k)| \lesssim_k 1+ V^k,
\]
so
\[
\|(\partial_x^2(V^k)) f\|_2 \lesssim_k \|f\|_2 + \|V^{k} f\| \lesssim_k \|(\opH^V)^{k+1} f\|_2,
\]
by part \ref{en:apriori1} and the induction hypothesis. Finally, again by the assumptions on $V$,
\[
|\partial_x (V^k)| \lesssim_k |x|^{\pdeg k-1}, \qquad V^k \gtrsim_k |x|^{\pdeg k},
\]
so (note that $\pdeg k - 1 > 0$ since $\pdeg > 1$  and $k \geq 1$),
\[
|\partial_x (V^k)| \lesssim_k \epsilon^{1-\pdeg k} + \epsilon V^k
\]
for all $\epsilon \in \Rpos$ (where the implied constants do not depend on $\epsilon$) and
\[\begin{split}
\|(\partial_x(V^k)) (\partial_x f)\|_2 &\lesssim_k \epsilon^{1-\pdeg k} \|\partial_x f\|_2 + \epsilon \|V^k \partial_x f\|_2 \\
&\leq \epsilon^{1-\pdeg k} \|(\opH^V)^{1/2} f\|_2 + \epsilon \|V^k \partial_x f\|_2\\
&\lesssim_k \epsilon^{1-\pdeg k} \|(\opH^V)^{k+1} f\|_2 + \epsilon \|V^k \partial_x f\|_2,
\end{split}\]
where parts \ref{en:apriori1} and \ref{en:apriori2} were used. Furthermore
\[
V^k \partial_x f = \partial_x(V^k f) - (\partial_x(V^k)) f
\]
and
\[\begin{split}
\|V^k \partial_x f\|_2 &\leq \|\partial_x(V^k f)\| + \|(\partial_x(V^k)) f\|_2 \\
&\leq\|(\opH^V)^{1/2} V^k f\|_2 + \|f\|_2 + \|V^k f\|_2 \\
&\lesssim_k \|(\opH^V) V^k f\|_2 + \|(\opH^V)^{k+1} f\|_2
\end{split}\]
where the induction hypothesis and parts \ref{en:apriori1} and \ref{en:apriori2} were used. Putting all together gives
\[
\|\opH^V (V^k f) \|_2 \lesssim_k (\epsilon+\epsilon^{1-\pdeg k}) \|(\opH^V)^{k+1} f\|_2 + \epsilon \|\opH^V (V^k f)\|_2;
\]
if $\epsilon \in \Rpos$ is now chosen sufficiently small, then the last term can be moved to the left-hand side and we obtain
\eqref{eq:induction_thesis2}.

Now, by part \ref{en:apriori3} and \eqref{eq:induction_thesis2},
\[
\|V^{k+1} f\|_2 \lesssim \|\opH^V (V^k f) \|_2 \lesssim_k \| (\opH^V)^{k+1} f\|_2
\]
and we are done.
\end{proof}

\section{One-parameter families of Schr\"odinger operators}\label{s:rescaledschroedinger}

Again, let $\pdeg,\co \in \Rpos$ be fixed, with $\pdeg > 1$.
For a given $V \in \Pot^\pdeg(\co)$, consider the family of operators $\opH^V_\scale \defeq \opH^{\scale V}$ associated to the potentials $\scale V$ as $\scale \in \Rpos$. Correspondingly we define
\[
E_n^V(\scale) \defeq E_n^{\scale V}, \qquad \psi_{n,\scale}^{V} \defeq \psi_n^{\scale V}, \qquad x_{n,\scale}^{V,\pm} \defeq x_n^{\scale V,\pm}.
\]

We collect in the next proposition a few important estimates on the objects above. The rest of the section is devoted to its proof.

\begin{prp}\label{eigenvalues-prp}
Let  $V \in \Pot^{\pdeg}(\co)$.
\begin{enumerate}[label=(\roman*)]
\item\label{en:eigenvalues1} For all $n \in \Npos$  and $\scale \in \Rpos$,
\[
E_n^V(\scale) \simeq \scale^{2/(2+\pdeg)} \, n^{2\pdeg/(2+\pdeg)}, \qquad\pm x_{n,\scale}^{V,\pm} \simeq \left(\frac{n}{\sqrt{\scale}}\right)^{2/(2+\pdeg)}.
\]
\item\label{en:eigenvalues-andep} $E_n^V(\scale)$ and $\psi_{n,\scale}^{V}$ depend analytically on $\scale \in \Rpos$.
\item\label{en:eigenvalues2} $E_n^V : \Rpos \to \Rpos$ is increasing, invertible and differentiable. Moreover,
\[
\partial_\scale E_n^V(\scale) \simeq \scale^{-1} E_n^V(\scale).
\]
\end{enumerate}
\end{prp}

\subsection{An eigenvalue counting formula for convex potentials and the proof of part \ref{en:eigenvalues1} of Proposition \ref{eigenvalues-prp}}

The proof of Proposition \ref{eigenvalues-prp} and, more crucially, that of Lemma \ref{lem:transition_est} below, relies on the following version of the Bohr--Sommerfeld formula.

\begin{thm}\label{titchmarsh-thm}
Let $V\in C(\RR)\cap C^2(\RR\setminus\{0\})$ be such that $V(0)=0$, $V'(x)>0$ for every $x>0$, $V'(x)<0$ for every $x<0$, and $V''(x)\geq0$ for every $x$. Then

\[
n = \frac{1}{\pi}\int_{x_n^{V,-}}^{x_n^{V,+}} \sqrt{E_n^V-V} + \BigO(1), 
\]
where the error term is bounded by an absolute constant (which may be taken to be $8+\frac{5}{2\pi^2}$).
\end{thm}
Notice that a potential satisfying the assumptions of Theorem \ref{titchmarsh-thm} 
is in $\GPot$,
so the quantities $x_n^{V,+}$, $x_n^{V,-}$, and $E_n^V$ are well-defined (and $E_n^V-V$ is positive on the interval of integration).

The proof of Theorem \ref{titchmarsh-thm} is essentially contained in Section $7.5$ of \cite{titchmarsh}, where the result is attributed to \cite{hartman_1952}. Since in neither reference there is an explicit discussion of the absolute nature of the error term, and the argument relies on several facts scattered through different sections of \cite{titchmarsh}, we devoted the Appendix to a discussion of that proof.

We now proceed to prove part \ref{en:eigenvalues1} of Proposition \ref{eigenvalues-prp}.
According to Remark \ref{rem:scaling}, if we choose $t = \scale^{-1/(\pdeg+2)}$, then the scaling of parameter $t$ maps $\scale V$ into a potential $W \in \Pot^\pdeg(\co)$, and 
\begin{equation}\label{eq:eigen-scale}
E_n^{V}(\scale) =E_n^{\scale V} = \scale^{2/(\pdeg+2)} E_n^{W}.
\end{equation}
On the other hand, since $W \in \Pot^{\pdeg}(\co)$, by comparison \eqref{eq:comparison},
\begin{equation}\label{eigen-comparison}
\co^{-1}E_n^{|\cdot|^\pdeg}\leq E^W_n\leq \co E_n^{|\cdot|^\pdeg}.
\end{equation}
Applying Theorem \ref{titchmarsh-thm} to the potential $|\cdot|^\pdeg$ (which is convex for $\pdeg>1$), we get
\[\begin{split}
n
&=  \frac{1}{\pi}\int_{-(E_n^{|\cdot|^\pdeg})^{1/\pdeg}}^{(E_n^{|\cdot|^\pdeg})^{1/\pdeg}} \sqrt{E_n^{|\cdot|^\pdeg}-x^{\pdeg}} \,dx + \BigO(1)\\
&=  (E_n^{|\cdot|^\pdeg})^{1/2+1/\pdeg}\frac{1}{\pi}\int_{-1}^1\sqrt{1-y^{\pdeg}} \,dy + \BigO(1).
\end{split}\]
If $n$ is larger than some universal constant, then this implies
\[
E_n^{|\cdot|^\pdeg}\simeq n^{2\pdeg/(2+\pdeg)},
\]
and this approximate identity trivially extends to all values of $n$, since $E_n^{|\cdot|^\pdeg} \geq E_1^{|\cdot|^\pdeg} > 0$ (just by our definition of the approximate equality sign). By \eqref{eq:eigen-scale} and \eqref{eigen-comparison}, we conclude that
\[
E^V_n(\scale) \simeq\scale^{2/(2+\pdeg)}n^{2\pdeg/(2+\pdeg)} \qquad\forall n \geq 1,\quad \forall \scale\in\Rpos.
\]
Since $\scale V(x_{n,\scale}^{V,\pm}) = E_n^V(\scale)$ and $V \in \Pot^\pdeg(\co)$, the approximate formula for the transition points follows immediately.

\subsection{Proof of parts \ref{en:eigenvalues-andep} and \ref{en:eigenvalues2} of Proposition \ref{eigenvalues-prp}}
By classical results of perturbation theory (see Chapter Seven of \cite{kato}, in particular Section $8$), the eigenvalues $E_n^V(\scale)$ and normalized eigenfunctions $\psi_{n,\scale}^V$ of $\opH_\scale^V$ depend analytically on the parameter $\scale$. In the case of eigenfunctions, this means that $\scale\mapsto\psi_{n,\scale}^V$ is analytic as an $L^2(\RR)$-valued mapping. In particular $\partial_\scale\psi_{n,\scale}^V$ and $\psi_{n,\scale}^V$ are orthogonal and, by differentiating both sides of
\[
E_n^V(\scale)=\langle (-\partial_x^2+\scale V)\psi_{n,\scale}^V ,\psi_{n,\scale}^V \rangle
\]
(where $\langle\cdot,\cdot\rangle$ is the scalar product in $L^2(\RR)$), we get:
\[
\partial_\scale(E_n^V)(\scale) = \int_\RR V|\psi_{n,\scale}^V|^2 + \langle \opH_\scale^V\partial_\scale\psi_{n,\scale}^V,\psi_{n,\scale}^V \rangle + \langle \opH_\scale^V\psi_{n,\scale}^V,\partial_\scale\psi_{n,\scale}^V \rangle = \int_\RR V|\psi_{n,\scale}^V|^2.
\]
This useful (and well-known) identity implies that $E_n^V(\scale)$ is increasing and that the desired estimate follows by Proposition \ref{virial-prp} applied to the potential $\scale V$.

\section{Analysis of Grushin operators and proof of Theorem \ref{thm:main}}\label{s:grushin}

Now that all the ingredients are in place, we proceed to the analysis of Grushin operators on $\RR^\done \times \RR^\dtwo$,
which will eventually lead to a proof of our multiplier theorem.
In what follows, with the symbols $\simeq$ and $\lesssim$ we denote inequalities with implicit constants depending only on the parameters $\hpdeg$, $\co$, $\done$, $\dtwo$ appearing in the statement of Theorem \ref{thm:main}.

\subsection{Preliminaries}

Let $\hpdeg \in (1/2,\infty)$ and $\co \in \Rpos$. Let $V_1,\dots,V_\done \in \Pot_{\even\convex}^{2\hpdeg}(\co)$. Define $V : \RR^\done \to \RR$ by $V(x) = \sum_{\jone=1}^\done V_j(x_j)$ and let the Grushin differential operator $\opL$ on $\RR^{\done + \dtwo}_z = \RR^{\done}_x \times \RR^{\dtwo}_y$ be defined by
\[
\opL = -\Delta_x - V(x) \Delta_y,
\]
where $\Delta_x = \sum_{\jone=1}^\done \partial_{x_\jone}^2$ and $\Delta_y = \sum_{\jtwo=1}^\dtwo \partial_{y_\jtwo}^2$.

This kind of operators is studied in \cite{robinson_analysis_2008} under much weaker assumptions on the function $V$. In particular, in \cite[Section 2]{robinson_analysis_2008} details are given on the definition of a self-adjoint extension of $\opL$ by means of the associated Dirichlet form. In addition, a number of properties of $\opL$ are obtained, in connection with the associated degenerate Riemannian geometry, which we summarise in the following statement.

\begin{prp}[\cite{robinson_analysis_2008}]\label{prp:doubling_gaussian}
There exists a distance $\dist$ on $\RR^2$ such that the following hold.
\begin{enumerate}[label=(\roman*)]
\item For all $z =(x,y)$ and $z'=(x',y') \in \RR^2$,
\[
\dist(z,z') \simeq |x-x'| + \min\left\{|y-y'|^{1/(1+\hpdeg)}, \frac{|y-y'|}{(|x|+|x'|)^{\hpdeg}}\right\}.
\]
\item If $\Vol(z,r)$ denotes the (Lebesgue) measure of the $\dist$-ball of centre $z=(x,y)$ and radius $r$, then
\[
\Vol(z,r) \simeq r^{\done+\dtwo} \max \{r,|x|\}^{\hpdeg\dtwo},
\]
so $\RR^{\done}\times \RR^{\dtwo}$ with the Lebesgue measure and the distance $\dist$ is a doubling metric measure space of homogeneous dimension $Q = \done+(1+\hpdeg)\dtwo$.
\item $\opL$ satisfies Gaussian-type heat kernel bounds relative to $\dist$, i.e., there exists $b>0$ such that
\[
|\Kern_{\exp(-t\opL)}(z,z')| \lesssim \Vol(z',t^{1/2})^{-1} \exp(-b \dist(z,z')^2/t),
\]
where $\Kern_{\exp(-t\opL)}$ denotes the integral kernel of the operator $\exp(-t\opL)$.
\end{enumerate}
\end{prp}
\begin{proof}
See \cite[Proposition 5.1 and Corollary 6.6]{robinson_analysis_2008}.
\end{proof}

Under our assumptions on $V$, the operator $\opL$ can be written as a sum,
\[
\opL = \sum_{\jone=1}^\done \opL_\jone,
\]
where
\[
\opL_\jone = -\partial_{x_\jone}^2 - V_\jone(x_\jone) \Delta_y.
\]
If we define the first-order differential operators
\[
T_\jtwo = -i\partial_{y_\jtwo},
\]
then the operators $\opL_1,\dots,\opL_\done,T_1,\dots,T_\dtwo$ commute pairwise. The joint spectral theory and functional calculus on $L^2$ for the above systems of commuting operators is conveniently described by means of a partial Fourier transform. Indeed, by taking the Fourier transform in the variable $y$, the operator $\opL_\jone$ corresponds to the family of Schr\"odinger operators in the variable $x_\jone$ defined by
\[
\opL_{\jone,\eta} = -\partial_{x_\jone}^2 + |\eta|^2 V_\jone(x_\jone)
\]
where $\eta \in \RR^\dtwo$. Similarly $T_\jtwo$ corresponds to the family of multiplication operators
\[
T_{\jtwo,\eta} = \eta_\jtwo.
\]

Note that $\opL_{\jone,\eta}$ is the Schr\"odinger operator $\opH^{V_\jone}_{|\eta|^2}$ of Section \ref{s:rescaledschroedinger}. In analogy with Section \ref{s:rescaledschroedinger}, for all $\scale \in \Rpos$, let $(E^\jone_n(\scale))_{n\in \Npos}$ denote the (increasing) sequence of eigenvalues of $\opH^{V_\jone}_\scale$ on $L^2(\RR)$, $\pm x_{n,\scale}^\jone$ denote the corresponding transition points (recall that $V_\jone$ is even),
  and $(\psi^{\jone}_{n,\scale})_{n\in \Npos}$ be the corresponding orthonormal sequence of real-valued eigenfunctions.
	
Now, for all $n \in \Npos^\done$, set
\[
\vec E_n(\scale) = (E^1_{n_1}(\scale),\dots,E^\done_{n_\done}(\scale))
\]
and 
\[
\psi_{n,\scale} = \psi^1_{n_1,\scale} \otimes \dots \otimes \psi^\done_{n_\done,\scale}
\]
for all $n \in \Npos^\done$.
If $\vec\opL$ and $\vec T$ denote the ``vectors of operators'' $(\opL_1,\dots,\opL_\done)$ and $(T_1,\dots,T_\dtwo)$ respectively, then, for all bounded Borel functions $F : \RR^{\done} \times \RR^{\dtwo} \to \CC$,
we can write
\[
F(\vec\opL,\vec T) f(z) = \int_{\RR^\done \times \RR^\dtwo} \Kern_{F(\vec\opL,\vec T)}(z,z') \, f(z') \,dz' ,
\]
where, for almost all $z=(x,y)$ and $z'=(x',y')$,
\[
\Kern_{F(\vec\opL,\vec T)}(z,z') = \frac{1}{(2\pi)^\dtwo} \int_{\RR^\dtwo} \sum_{n \in \Npos^\done} F(\vec E_n(|\eta|^2),\eta) \, \psi_{n,|\eta|^2}(x) \, \psi_{n,|\eta|^2}(x') \, e^{i\eta \cdot (y-y')} \, d\eta.
\]
Orthonormality of the eigenfunction systems and the Plancherel formula for the Fourier transform then yield
\begin{equation}\label{eq:exact_plancherel}
\|\Kern_{F(\vec\opL,\vec T)}(\cdot,z')\|_{L^2(\RR^{\done+\dtwo})}^2 = \frac{1}{(2\pi)^\dtwo} \int_{\RR^\dtwo} \sum_{n \in \Npos^\done} |F(\vec E_n(|\eta|^2),\eta)|^2 \, |\psi_{n,|\eta|^2}(x')|^2 \, d\eta
\end{equation}
for almost all $z' = (x',y') \in \RR^{\done+\dtwo}$.

In particular, if we restrict to the joint functional calculus of $\opL = \opL_1 + \dots + \opL_\done$ and $|\vec T|^2 = T_1^2 + \dots +T_\done^2$, and we define
\[
\ESum_n(\scale) = \sum_{\jone=1}^\done E^\jone_{n_\jone}(\scale)
\]
for all $n \in \Npos^\done$ and $\scale\in\Rpos$, 
then the above Plancherel-type identity \eqref{eq:exact_plancherel} simplifies as follows.

\begin{prp}\label{prp:plancherel_radial}
For all bounded Borel functions $F : \RR^2 \to \CC$,
\begin{equation}\label{eq:exact_plancherel_rad}
\|\Kern_{F(\opL,|\vec T|^2)}(\cdot,z')\|_{L^2(\RR^{\done+\dtwo})}^2 
= C_\dtwo \int_{0}^\infty \sum_{n \in \Npos^\done} |F(\ESum_n(\scale),\scale)|^2 \, |\psi_{n,\scale}(x')|^2 \, \scale^{\dtwo/2-1} \, d\scale ,
\end{equation}
for almost all $z' = (x',y') \in \RR^{\done+\dtwo}$. Here $C_{\dtwo} \in \Rpos$ is a suitable constant.
\end{prp}

\subsection{A weighted Plancherel estimate}\label{ss:weighted}

From the Plancherel-type identity of Proposition \ref{prp:plancherel_radial}, we now derive a weighted estimate for $\Kern_{F(\opL)}$.

For all $\gamma \in \Rnon$ and $\jone = 1,\dots,\done$, let $\Mult_\jone^\gamma,\Mult^\gamma : L^2(\RR^{\done \times \dtwo}) \to L^2(\RR^{\done \times \dtwo})$ denote the multiplication operators defined by
\[
\Mult_\jone^\gamma f(z) = |x_j|^\gamma f(z), \qquad 
\Mult^\gamma f(z) = |x|^\gamma f(z)
\]
for all $z = (x,y) \in \RR^\done \times \RR^\dtwo$ and $f \in L^2(\RR^\done \times \RR^\dtwo)$.

By Proposition \ref{eigenvalues-prp}, for all $n \in \Npos^\done$ and $\jone=1,\dots,\done$, the functions $E^\jone_{n_\jone} : \Rpos \to \Rpos$ are continuously differentiable, increasing and invertible, so their sum $\ESum_n : \Rpos \to \Rpos$ is too; let us denote by $\FScale_n : \Rpos \to \Rpos$ the inverse of $\ESum_n$.

\begin{prp}\label{prp:weighted1}
For all $\gamma \in \Rnon$, all bounded Borel functions $F : \RR \to \CC$, and all $z' = (x',y') \in \RR^{\done+\dtwo}$,
\begin{multline}\label{eq:weighted1}
\| W^\gamma \Kern_{F(\opL)}(\cdot,z')\|_{L^2(\RR^{\done+\dtwo})}^2 \\
\lesssim_\gamma \int_0^\infty |F(\lambda)|^2 \sum_{n \in \Npos^\done} \frac{\lambda^{\gamma/\hpdeg+1}}{\FScale_n(\lambda)^{\gamma/\hpdeg+1-\dtwo/2}} \, |\psi_{n,\FScale_n(\lambda)}(x')|^2 \, \FScale_n'(\lambda) \, \frac{d\lambda}{\lambda}.
\end{multline}
\end{prp}
\begin{proof}

Let
\[
f^\eta(x) = \int_{\RR^\dtwo} f(x,y) \, e^{-i \eta \cdot y} \,dy
\]
denote the partial Fourier transform of $f \in L^2(\RR^{\done}_x \times \RR^{\dtwo}_y)$ in the $y$-variable. Then, for all $N \in \NN$  and $\jone=1,\dots,\done$,
\begin{multline*}
\|\opL_\jone^N f\|_2^2 = \frac{1}{(2\pi)^\dtwo} \int_{\RR^\dtwo} \|(\opL_{\jone,\eta})^N f^\eta\|_2^2 \,d\eta \\
 \gtrsim_N \frac{1}{(2\pi)^\dtwo} \int_{\RR^\dtwo} \|(|\eta|^2 V_\jone)^N f^\eta\|_2^2 \,d\eta = \| V_\jone^N |\vec T|^{2N}f\|_2^2,
\end{multline*}
by Proposition \ref{apriori_scaled-prp}. Hence
\[
\|W_\jone^{2N\hpdeg} f\|_2 \lesssim_N \|\opL_\jone^N |\vec T|^{-2N} f\|_2
\]
for all $N \in \NN$ and $\jone=1,\dots,\done$, and therefore
\[
\|W^{2N\hpdeg} f\|_2 \lesssim_N \| \opL^N|\vec T|^{-2N} f\|_2
\]
for all $N \in \NN$. Finally, by interpolation,
\begin{equation}\label{eq:weight_control}
\| W^\gamma f \|_2 \lesssim_\gamma \| \opL^{\gamma/(2\hpdeg)} |\vec T|^{-\gamma/\hpdeg} f\|_2
\end{equation}
for all $\gamma \in \Rnon$.

By combining \eqref{eq:exact_plancherel_rad} and \eqref{eq:weight_control}, we obtain
\[\begin{split}
&\| W^\gamma \Kern_{F(\opL)}(\cdot,z')\|_{L^2(\RR^{\done+\dtwo})}^2 \\
&\lesssim_\gamma \| \opL^{\gamma/(2\hpdeg)} |\vec T|^{-\gamma/\hpdeg} \Kern_{F(\opL)}(\cdot,z')\|_{L^2(\RR^{\done+\dtwo})}^2 \\
&= \|  \Kern_{ \opL^{\gamma/(2\hpdeg)} |\vec T|^{-\gamma/\hpdeg} F(\opL)}(\cdot,z')\|_{L^2(\RR^{\done+\dtwo})}^2 \\
&\simeq_\gamma \int_0^\infty \sum_{n \in \Npos^\done} |F(\ESum_n(\scale))|^2 \, |\ESum_n(\scale)/\scale|^{\gamma/\hpdeg} \, |\psi_{n,\scale}(x')|^2 \,\scale^{\dtwo/2-1}\, d\scale.
\end{split}\]
Now, since $\ESum_n : \Rpos \to \Rpos$ is increasing and invertible and its inverse $\FScale_n : \Rpos \to \Rpos$ is continuously differentiable, we can use the change of variable $\lambda = \ESum_n(\scale)$ in the last integral and obtain the conclusion.
\end{proof}

The right-hand side of \eqref{eq:weighted1} can be thought of as the $L^2$-norm of $F$ with respect to a (weighted) ``Plancherel measure'', whose density with respect to the Lebesgue measure on $\Rpos$ is expressed in terms of eigenvalues and eigenfunctions of Schr\"odinger operators. We want now to obtain a precise estimate of this density, by means of the bounds obtained in Sections \ref{s:schroedinger} and \ref{s:rescaledschroedinger}.

We first rewrite in a more convenient form the previously obtained estimates for eigenvalues and transition points (which in turn enter into estimates for eigenfunctions). It will be convenient to denote by
\[
\tilde x_{n,\lambda} = (x^1_{n_1,\FScale_n(\lambda)},\dots,x^{\done}_{n_\done,\FScale_n(\lambda)})
\]
the vector of transition points corresponding to $n \in \Npos^\done$ and $\lambda \in \Rpos$. As another application of Theorem \ref{titchmarsh-thm}, we also obtain an estimate for ``gaps'' between transition points that will play an important role in what follows.

\begin{lem}\label{lem:transition_est}
For all $\lambda \in \Rpos$ and $n \in \Npos^\done$,
\begin{equation}
\label{eq:inv_eigen} \FScale_{n}(\lambda) \simeq \lambda \FScale'_{n}(\lambda) \simeq \lambda^{1+\hpdeg} \, |n|^{-2\hpdeg}.
\end{equation}
Moreover, for all $\lambda \in \Rpos$, $n \in \Npos^\done$, $\jone \in \{1,\dots,\done\}$,
\begin{gather}
\label{eq:trans_est1} \lambda^{1/2} (\tilde x_{n,\lambda})_\jone \simeq |n|^{\hpdeg/(1+\hpdeg)} \, n_\jone^{1/(1+\hpdeg)},\\
\label{eq:trans_est2} \lambda^{1/2} |\tilde x_{n,\lambda}| \simeq |n|.
\end{gather}
In addition there exists a universal constant $\uK \in \Npos$ such that, for all $\lambda \in \Rpos$ and all distinct $n,n' \in \Npos^\done$,
\begin{equation}\label{eq:gap}
\lambda^{1/2} \max\{ |(\tilde x_{n',\lambda})_\jone - (\tilde x_{n,\lambda})_\jone| \tc \jone=1,\dots,\done, \, n_\jone \neq n'_\jone \} \gtrsim 1
\end{equation}
whenever
\begin{equation}\label{eq:gap_ass}
\min(\{|n_\jone-n'_\jone| \tc \jone=1,\dots,\done, \, n_\jone \neq n'_\jone \}) \geq \uK.
\end{equation}
\end{lem}
\begin{proof}
By Proposition \ref{eigenvalues-prp},
\[
E^\jone_{n_\jone}(\scale) \simeq \scale \partial_\scale E^\jone_{n_\jone}(\scale) \simeq \scale^{1/(1+\hpdeg)} \, n_\jone^{2\hpdeg/(1+\hpdeg)},
\]
hence
\[
\ESum_{n}(\scale) \simeq \scale \partial_\scale \ESum_{n}(\scale) \simeq \scale^{1/(1+\hpdeg)} \, |n|^{2\hpdeg/(1+\hpdeg)},
\]
and \eqref{eq:inv_eigen}
follows, since $\FScale_n$ is the inverse of $\ESum_n$. This estimate, combined again with Proposition \ref{eigenvalues-prp}, yields \eqref{eq:trans_est1} and \eqref{eq:trans_est2}.

We know that, for fixed $n \in \Npos$, the function $\ESum_n$ and $\FScale_n$ are strictly increasing. Note now that, if $n\leq n'$ componentwise, then $\ESum_n(\scale) \leq \ESum_{n'}(\scale)$  for all $\scale \in \Rpos$; hence, for all $\lambda \in \Rpos$, $\FScale_{n'}(\lambda) \leq \FScale_{n}(\lambda)$ whenever $n\leq n'$.

Recall that the transition points $(\tilde x_{n,\lambda})_\jone \in \Rpos$ are defined by
\[
\FScale_n(\lambda) V_\jone((\tilde x_{n,\lambda})_\jone) = E_{n_\jone}^j(\FScale_n(\lambda))
\]
and (since the potentials $V_\jone$ are even and convex) Theorem \ref{titchmarsh-thm} yields that
\begin{equation}\label{eq:bohrsommerfeld_app}
n_\jone = \FScale_n(\lambda)^{1/2} H_\jone((\tilde x_{n,\lambda})_\jone) + \BigO(1),
\end{equation}
where $|\BigO(1)| \leq (\uK-1)/2$ for some universal constant $\uK \in \Npos$ and
\begin{equation}\label{eq:reduced-bohrsommerfeld}
H_\jone(t) = \frac{2}{\pi} \int_0^t (V_\jone(t)-V_\jone(s))^{1/2} \,ds \qquad (t\in\Rpos).
\end{equation}
Note that $H_\jone$ is continuously differentiable and
\[
H_\jone'(t) = \frac{1}{\pi} \int_0^t \frac{V_\jone'(t)}{(V_\jone(t)-V_\jone(s))^{1/2}} \,ds = \frac{1}{\pi} \int_0^{V_\jone(t)} \frac{V_\jone'(t) (V_\jone^{-1})'(v)}{(V_\jone(t)-v)^{1/2}} \,dv.
\]
Recall that, since $V_\jone \in \Pot^{2\hpdeg}(\co)$, for all $t \in \Rpos$,
\[
V_\jone(t) \simeq t V_\jone'(t) \simeq t^{2\hpdeg},
\]
hence, for all $v \in \Rpos$,
\[
V_\jone^{-1}(v) \simeq v (V_\jone^{-1})'(v) \simeq v^{1/(2\hpdeg)}
\]
and
\begin{equation}\label{eq:der_bohrsommerfeld}\begin{split}
H_\jone'(t) &\simeq \frac{V_\jone(t)}{t} \int_0^{V_\jone(t)} \frac{v^{1/(2\hpdeg)-1}}{(V_\jone(t)-v)^{1/2}} \,dv \\
&= \frac{V_\jone(t)^{1/2+1/(2\hpdeg)}}{t} \int_0^{1} \frac{v^{1/(2\hpdeg)-1}}{(1-v)^{1/2}} \,dv \simeq V_\jone(t)^{1/2}.
\end{split}\end{equation}

From \eqref{eq:bohrsommerfeld_app} we then have, for all $n,n' \in \Npos^\done$,
\[\begin{split}
n'_\jone - n_\jone + \BigO(1) 
&= \FScale_{n'}(\lambda)^{1/2} H_\jone((\tilde x_{n',\lambda})_\jone) - \FScale_n(\lambda)^{1/2} H_\jone((\tilde x_{n,\lambda})_\jone) \\
&= (\FScale_{n'}(\lambda)^{1/2}-\FScale_{n}(\lambda)^{1/2}) H_\jone((\tilde x_{n',\lambda})_\jone) \\&\qquad+ \FScale_n(\lambda)^{1/2} (H_\jone((\tilde x_{n',\lambda})_\jone) - H_\jone((\tilde x_{n,\lambda})_\jone)),
\end{split}\]
where now $|\BigO(1)| \leq \uK-1$.

Assume now that $n,n'$ satisfy the assumptions \eqref{eq:gap_ass}. Then we claim that, up to switching $n$ and $n'$, we can find $\jone \in \{1,\dots,\done\}$ so that
\[
n'_\jone - n_\jone \geq \uK \qquad\text{and}\qquad \FScale_{n'}(\lambda) \leq \FScale_{n}(\lambda).
\]
Indeed, this is clear when $n'-n$ has two nonzero components of opposite signs, since in this case one can choose the component with the same sign as $\FScale_{n}(\lambda)-\FScale_{n'}(\lambda)$. Otherwise, up to switching, $n' \geq n$ componentwise, but then $\FScale_{n'}(\lambda) \leq \FScale_{n}(\lambda)$ and any nonzero component of $n'-n$ will do.

Under these assumptions, $\FScale_{n'}(\lambda)^{1/2}-\FScale_{n}(\lambda)^{1/2} \leq 0$ and therefore
\begin{equation}\label{eq:gap_lb_bohrsommerfeld}\begin{split}
1 \leq n'_\jone - n_\jone + \BigO(1) 
&\leq \FScale_n(\lambda)^{1/2} (H_\jone((\tilde x_{n',\lambda})_\jone) - H_\jone((\tilde x_{n,\lambda})_\jone)) \\
&= \FScale_n(\lambda)^{1/2} \, H_\jone'(s) \, ((\tilde x_{n',\lambda})_\jone - (\tilde x_{n,\lambda})_\jone)
\end{split}\end{equation}
for some $s$ between $(\tilde x_{n,\lambda})_\jone$ and $(\tilde x_{n',\lambda})_\jone$. 

Since $H_j$ is increasing, the above inequality shows that $(\tilde x_{n',\lambda})_\jone > (\tilde x_{n,\lambda})_\jone$.
Moreover, if $(\tilde x_{n',\lambda})_\jone \geq 2(\tilde x_{n,\lambda})_\jone$, then inequality \eqref{eq:trans_est1} gives
\[
\lambda^{1/2} ((\tilde x_{n',\lambda})_\jone - (\tilde x_{n,\lambda})_\jone) \geq \lambda^{1/2} (\tilde x_{n,\lambda})_\jone \gtrsim 1;
\]
if instead $(\tilde x_{n',\lambda})_\jone \leq 2(\tilde x_{n,\lambda})_\jone$, then $(\tilde x_{n,\lambda})_\jone \simeq (\tilde x_{n',\lambda})_\jone \simeq s$, hence, by \eqref{eq:der_bohrsommerfeld},
\[
\FScale_n(\lambda)^{1/2} H_\jone'(s) \simeq \FScale_n(\lambda)^{1/2} V_\jone((\tilde x_{n,\lambda})_\jone)^{1/2} = E^\jone_{n_\jone}(\FScale_n(\lambda))^{1/2} \leq \Sigma_n(\FScale_n(\lambda))^{1/2}\leq \lambda^{1/2}
\]
and again, by \eqref{eq:gap_lb_bohrsommerfeld},
\[
\lambda^{1/2} ((\tilde x_{n',\lambda})_\jone - (\tilde x_{n,\lambda})_\jone)
\gtrsim 1.
\]
In any case, \eqref{eq:gap} is proved.
\end{proof}

\begin{rem}
It is in the above proof of the ``gap estimates'' for transition points that the parity assumption on the potentials $V_\jone$ is essentially used, in that it allows one to precisely relate positive and negative transition points, and the corresponding gaps. Indeed, the integral relation in Theorem \ref{titchmarsh-thm} a priori would appear to provide information only on the sum of the gaps between positive transition points and between negative transition points, but not on the two gaps separately. One could however somehow ``relax'' the parity assumption by requiring the potentials $V_\jone$ to satisfy
\[
V_\jone(-t) = V_\jone( \theta_\jone t)
\]
for some constants $\theta_\jone \in \Rpos$ and all $t \in \Rpos$; indeed, all the results of Section \ref{s:grushin} could be obtained, mutatis mutandis, under this more general ``skewed parity'' assumption.
\end{rem}

Define $|y|_\infty \defeq \max_j|y_j|$ for $y \in \RR^\done$. We now rewrite in a convenient form the estimates for eigenfunctions.

\begin{lem}\label{lem:ef_estimates}
For all $\lambda \in \Rpos$, $n \in \Npos^\done$, $x \in \RR^\done$ and $\jone \in \{1,\dots,\done\}$,
\begin{equation}\label{eq:ef_trans}
\lambda^{-1/2} |\psi^\jone_{n_\jone,\FScale_n(\lambda)}(x_\jone)|^2 \lesssim 
(\lambda^{1/2} (\tilde x_{n,\lambda})_\jone)^{-1/2} (1 + ||\lambda^{1/2} x_\jone| - \lambda^{1/2} (\tilde x_{n,\lambda})_\jone |)^{-1/2}.
\end{equation}
Moreover there exists $\kappa \simeq 1$ such that, for all $N \in \NN$,
\begin{equation}\label{eq:ef_exp_joint}
\lambda^{-\done/2} |\psi_{n,\FScale_n(\lambda)}(x)|^2 \lesssim_N (\lambda^{1/2} |x|)^{-N}
\end{equation}
whenever $|x|_\infty \geq \kappa  |\tilde x_{n,\lambda}|_\infty$.
\end{lem}
\begin{proof}
From Proposition \ref{eigenfunctions-prp} it follows that, for all $x_\jone \in \RR$,
\[
|\psi^\jone_{n_\jone,\FScale_n(\lambda)}(x_\jone)|^2 
\lesssim (\tilde x_{n,\lambda})_\jone^{-1/2} ||x_\jone|-(\tilde x_{n,\lambda})_\jone|^{-1/2} 
\]
and also (see Remark \ref{rem:weaker-unif})
\[
|\psi^\jone_{n_\jone,\FScale_n(\lambda)}(x_\jone)|^2 
\lesssim E^{\jone}_{n_\jone}(\FScale_n(\lambda))^{1/4} (\tilde x_{n,\lambda})_\jone^{-1/2}
\leq \lambda^{1/4} (\tilde x_{n,\lambda})_\jone^{-1/2},
\]
whence \eqref{eq:ef_trans} follows. In particular, since $\lambda^{1/2} (\tilde x_{n,\lambda})_\jone \gtrsim 1$ (see \eqref{eq:trans_est1}), we also obtain
\begin{equation}\label{eq:ef_unif}
\lambda^{-1/2} |\psi^\jone_{n_\jone,\FScale_n(\lambda)}(x_\jone)|^2 \lesssim 1.
\end{equation}

Moreover, if $|x_\jone| \geq \kappa (\tilde x_{n,\lambda})_j$, then
 \eqref{eq:inv_eigen} and 
Proposition \ref{eigenfunctions-prp} yield that
\begin{equation}\label{eq:ef_exp}\begin{split}
|\psi^\jone_{n_\jone,\FScale_n(\lambda)}(x_\jone)|^2
&\lesssim E^{\jone}_{n_\jone}(\FScale_n(\lambda))^{1/2} 
 \exp(-\delta |x_\jone| (\FScale_n(\lambda) V_\jone(x_\jone))^{1/2}) \\
&\lesssim \lambda^{1/2}
 \exp(-\delta' |n|^{-\hpdeg} (\lambda^{1/2} |x_\jone|)^{\hpdeg+1}) \\
& \lesssim_N \lambda^{1/2} (|n|^{-\hpdeg} (\lambda^{1/2} |x_\jone|)^{\hpdeg+1})^{-N}
\end{split}\end{equation}
for all $N \in \NN$ (here $\delta'\simeq1$).
In particular, if $|x|_\infty \geq \kappa |\tilde x_{n,\lambda}|_\infty$, then there exist $\jone_0$ so that $|x|_\infty = |x_{\jone_0}|$, and \eqref{eq:trans_est2} gives
\[
\lambda^{1/2} |x_{\jone_0}| = \lambda^{1/2} |x|_\infty \gtrsim \lambda^{1/2} |\tilde x_{n,\lambda}|_\infty \simeq |n|.
\]
therefore
\[
|n|^{-\hpdeg} (\lambda^{1/2} |x_{\jone_0}|)^{\hpdeg+1} \gtrsim \lambda^{1/2} |x|_\infty
\]
and \eqref{eq:ef_exp_joint} follows by combining \eqref{eq:ef_exp} for $\jone =\jone_0$ and \eqref{eq:ef_unif} for $\jone \neq \jone_0$.
\end{proof}

In applying the above estimate \eqref{eq:ef_trans}, we will be interested in controlling a multivariate sum with the corresponding integral (in order to exploit the integrability of the inverse-square-root singularity). The following lemma, giving sufficient conditions for such a control to be valid, explains the importance of the previously discussed ``gap estimates'' (see also \cite[p.\ 1285]{chen_sharp_2013} and \cite[Lemma 4.1]{casarino_grushinsphere}).

\begin{lem}\label{lem:sum_integral}
Let $\kappa \in [1,\infty)$. Let $\Omega \subseteq \RR^d$ be open and convex and $H : \Omega \to \Rpos$ be locally Lipschitz and satisfying
\[
|\nabla H(\vu)| \leq \kappa H(\vu)
\]
for almost all $\vu \in \Omega$. Let $P \subseteq \Omega$ be such that, for some $r \in (0,1]$,
\[
\inf_{\vu \in P} |B_r(\vu) \cap \Omega| \geq \kappa^{-1}
\]
(here $B_r(\vu)$ is the ball centered at $\vu$ of radius $r$) and moreover we can decompose $P = P_1 \cup \dots \cup P_N$ for some $N \leq \kappa$ so that
\[
\inf_{j=1,\dots,\kappa} \inf_{\substack{\vu,\vu' \in P_j\\ \vu\neq \vu' }} |\vu-\vu'| \geq 2r.
\]
Then
\[
\sum_{\vu \in P} H(\vu) \leq e\kappa^3 \int_\Omega H(\vu) \,dx.
\]
\end{lem}
\begin{proof}
From the differential inequality and the convexity of $\Omega$ we obtain that
\[
H(\vu) \leq \kappa \exp(|\vu-\vu'|) H(\vu')
\]
for all $\vu,\vu' \in \Omega$. The lower bound on distances of points of $P_j$ implies that the sets $B_r(\vu) \cap \Omega$ for $\vu \in P_j$ are pairwise disjoint, so
\begin{multline*}
\sum_{\vu \in P_j} H(\vu) \leq \kappa \sum_{\vu \in P_j} |B_r(\vu) \cap \Omega| H(\vu) \\
\leq \kappa^2 \exp(r) \sum_{\vu \in P_j} \int_{B_r(\vu) \cap \Omega} H(\vu') \,d\vu' \leq e \kappa^2 \int_\Omega H(\vu') \,d\vu'
\end{multline*}
and the conclusion follows since $\sum_{\vu \in P} = \sum_{j=1}^N \sum_{\vu \in P_j}$ and $N\leq \kappa$.
\end{proof}

We can now estimate the density of the ``Plancherel measure'' in \eqref{eq:weighted1} as follows.

\begin{prp}\label{keytoweighted-prp}
Let $\gamma \in [0,\dtwo\hpdeg/2)$. For all $\lambda \in \Rpos$ and $x \in \RR$,
\[
\max\{\lambda^{-1/2},|x|\}^{\dtwo\hpdeg-2\gamma} \sum_{n \in \Npos^\done} \frac{\lambda^{\gamma/\hpdeg+1-(\done+\dtwo)/2}}{\FScale_n(\lambda)^{\gamma/\hpdeg+1-\dtwo/2}} \, |\psi_{n,\FScale_n(\lambda)}(x)|^2 \, \FScale_n'(\lambda) \lesssim_\gamma 1.
\]
\end{prp}
\begin{proof}
In view of \eqref{eq:inv_eigen},
the estimate to be proved is equivalent to
\begin{equation}\label{eq:sumtocontrol}
\max\{1,\lambda^{1/2} |x|\}^{\varepsilon} \sum_{n \in \Npos^\done} |n|^{-\varepsilon} \, \lambda^{-\done/2} \, |\psi_{n,\FScale_n(\lambda)}(x)|^2 \lesssim_\varepsilon 1,
\end{equation}
where we have set $\varepsilon = \dtwo\hpdeg-2\gamma\in (0,\dtwo\hpdeg]$.

We consider first the part of the sum in \eqref{eq:sumtocontrol} where $|x|_\infty \geq \kappa |\tilde x_{n,\lambda}|_\infty$. By \eqref{eq:trans_est2}, this condition is equivalent to $\lambda^{1/2}|x|\gtrsim |n|$. In particular, it is empty unless $\lambda^{1/2}|x| \gtrsim 1$. Therefore, using \eqref{eq:ef_exp_joint} and the positivity of $\varepsilon$, this part of the sum is controlled by
\[
(\lambda^{1/2} |x|)^{\varepsilon} \sum_{\substack{n \in \Npos^\done\\ \lambda^{1/2}|x|\gtrsim |n|}}  (\lambda^{1/2} |x|)^{-N}
 \lesssim_\varepsilon (\lambda^{1/2} |x|)^{\varepsilon+\done-N} \lesssim_\varepsilon 1,
\]
by choosing $N\geq \done+\varepsilon$.

For the remaining part of the sum, we use a different estimate. 
Namely, from \eqref{eq:ef_trans} it follows that
\[
\lambda^{-1/2} |\psi^\jone_{n_\jone,\FScale_n(\lambda)}(x_\jone)|^2 
\lesssim
\Phi_{\lambda^{1/2} x_\jone}((\lambda^{1/2} \tilde x_{n,\lambda})_\jone),
\]
where
\[
\Phi_t(\vu) = \vu^{-1/2} (1+||t|-\vu|)^{-1/2}
\]
for all $t \in \RR$ and $\vu \in \Rpos$. So we are reduced to proving that
\begin{equation}\label{eq:remainingsum}
\max\{1,\lambda^{1/2} |x|\}^\varepsilon \sum_{\substack{n \in \Npos^\done \\ |x|_\infty \leq \kappa |\tilde x_{n,\lambda}|_\infty}} |n|^{-\varepsilon} 
\prod_{\jone=1}^\done \Phi_{\lambda^{1/2} x_\jone}((\lambda^{1/2} \tilde x_{n,\lambda})_\jone)\lesssim_\varepsilon 1
\end{equation}

It should be noted that, for all $c \in \Rpos$, the derivative $\Phi'_t$ of $\Phi_t$ satisfies
\begin{equation}\label{eq:gradient_estimate}
|\Phi_t'(\vu)| \lesssim_c \Phi_t(\vu)
\end{equation}
uniformly in $\vu \in [c,\infty)$ and $t \in \RR$. Moreover, by \eqref{eq:trans_est1}, there exists $c \in \Rpos$ so that
\[
\lambda^{1/2} (\tilde x_{n,\lambda})_\jone \in [c,\infty)
\]
for all $\lambda \in \Rpos$, $n \in \Npos^\done$ and $j \in \{1,\dots,\done\}$.

For all $J \subseteq \{1,\dots,\done\}$, define
\[
N^x_\lambda(J) = \{n \in \Npos^\done \tc (\tilde x_{n,\lambda})_\jone \leq 2 |x_\jone| \text{ if $\jone \in J$, } (\tilde x_{n,\lambda})_\jone > 2 |x_\jone| \text{ if $j \notin J$}\};
\]
in addition, for all $\jone_0 \in \{1,\dots,\done\}$, define
\[
M(\jone_0) = \{ n \in \Npos^\done \tc |n|_\infty = n_{\jone_0} > \max\{ n_\jone \tc \jone < \jone_0\}\};
\]
then we can split the sum in \eqref{eq:remainingsum} as follows:
\[
\sum_{\substack{n \in \Npos^\done \\ |x|_\infty \leq \kappa |\tilde x_{n,\lambda}|_\infty}} = \sum_{\substack{\jone_0 \in \{1,\dots,\done\} \\J \subseteq \{1,\dots,\done\} }} \sum_{\substack{n \in N_\lambda^x(J) \cap M(\jone_0) \\ |x|_\infty \leq \kappa |\tilde x_{n,\lambda}|_\infty}}.
\]
This splitting is convenient in that it allows us to identify the largest component $n_{\jone_0}$ of the multiindex $n$, as well as to distinguish between the components $(\tilde x_{n,\lambda})_\jone$ of the ``transition vector'' according to whether they stay near ($j\in J$) or far ($j \notin J$) the corresponding $|x_\jone|$.

Let now $J \subseteq \{1,\dots,\done\}$ and $\jone_0 \in \{1,\dots,\done\}$ be given, and set $J^c = \{1,\dots,\done\} \setminus J$. Note that
\[
\Phi_{t}(\vu) \lesssim \vu^{-1}
\]
for all $\vu \in [2|t|,\infty)$, uniformly in $t \in \RR$.
Then, by \eqref{eq:trans_est1},
\begin{equation}\label{eq:remaining_sum_explicit}
\begin{split}
&\max\{1,\lambda^{1/2} |x|\}^\varepsilon \sum_{\substack{n \in N_\lambda^x(J) \cap M(\jone_0) \\ |x|_\infty \leq \kappa |\tilde x_{n,\lambda}|_\infty}} |n|^{-\varepsilon} 
\prod_{\jone=1}^\done \Phi_{\lambda^{1/2} x_\jone}((\lambda^{1/2} \tilde x_{n,\lambda})_\jone)\\
&\lesssim_\varepsilon \max\{1,\lambda^{1/2} |x|\}^\varepsilon \sum_{\substack{n \in N_\lambda^x(J) \cap M(\jone_0) \\ |x|_\infty \leq \kappa |\tilde x_{n,\lambda}|_\infty}} n_{\jone_0}^{-\varepsilon} \prod_{\jone \in J^c} \frac{n_\jone^{-1/(1+\hpdeg)}}{n_{\jone_0}^{\hpdeg/(1+\hpdeg)}} 
\prod_{\jone \in J} \Phi_{\lambda^{1/2} x_\jone}((\lambda^{1/2} \tilde x_{n,\lambda})_\jone) .
\end{split}
\end{equation}

Suppose first that $\jone_0 \notin J$ and define $J' = J^c \setminus \{\jone_0\}$. Then the above quantity is controlled by
\begin{equation}\label{eq:sum_max_no_J}
\begin{split}
&\max\{1,\lambda^{1/2} |x|\}^\varepsilon \sum_{\substack{n_{\jone_0} \in \Npos \\ n_{\jone_0} \gtrsim \lambda^{1/2}|x|}} n_{\jone_0}^{-\varepsilon-1} \sum_{\substack{(n_\jone)_{\jone \in J'} \in \Npos^{J'} \\ n_\jone \leq n_{\jone_0} \,\forall \jone \in J'}} \prod_{\jone \in J'} \frac{n_\jone^{-1/(1+\hpdeg)}}{n_{\jone_0}^{\hpdeg/(1+\hpdeg)}} \\
 &\qquad\times \sum_{\substack{(n_\jone)_{\jone \in J} \in \Npos^J \\ (\lambda^{1/2} \tilde x_{n,\lambda})_\jone \leq 2|\lambda^{1/2} x_\jone|}}
\prod_{\jone \in J} \Phi_{\lambda^{1/2} x_\jone}((\lambda^{1/2} \tilde x_{n,\lambda})_\jone).
\end{split}
\end{equation}
In the above expression, the multiindex $n = (n_\jone)_{\jone\in\{1,\dots,\done\}}$ must be thought of as composed of three ``independent'' parts $n_{\jone_0}$, $(n_\jone)_{\jone \in J'}$ and $(n_\jone)_{\jone \in J}$, which serve as summation (multi)indices of three different sums.

In order to control the inner sum, for each fixed $n_{\jone_0} \in \Npos$ and $(n_\jone)_{\jone \in J'} \in \Npos^{J'}$ we apply Lemma \ref{lem:sum_integral} to the function $H = \bigotimes_{\jone \in J} \Phi_{\lambda^{1/2} x_\jone}$ and the set $P = \{((\lambda^{1/2} \tilde x_{n,\lambda})_\jone)_{\jone \in J}\}_{(n_\jone)_{\jone \in J} \in \Npos^J}$. Note that, if $\uK \in \Npos$ is the constant given by Lemma \ref{lem:transition_est} and we split $P = \bigcup_{m \in \{0,\dots,\uK-1\}^J} P_{m}$, where
\[
P_{m} = \{((\lambda^{1/2} \tilde x_{n,\lambda})_\jone)_{\jone \in J} \tc (n_\jone)_{\jone \in J} \in \Npos^J, \, n_\jone \equiv m_\jone \text{ (mod $\uK$) } \, \forall  \jone \in J \},
\]
then \eqref{eq:gap} implies that
\[
\inf_{\substack{\vu,\vu' \in P_m\\ \vu \neq \vu'}} |\vu-\vu'| \gtrsim 1 
\]
uniformly in $m$, $x$, $(n_\jone)_{\jone \in J^c}$ and $\lambda$. Hence, also in view of \eqref{eq:gradient_estimate}, the assumptions of Lemma \ref{lem:sum_integral} are satisfied if we take $\Omega = \prod_{\jone \in J} (c/2,4|\lambda^{1/2} x_\jone|)$, and 
\eqref{eq:sum_max_no_J} is majorized by
\[\begin{split}
&\max\{1,\lambda^{1/2} |x|\}^\varepsilon \sum_{\substack{n_{\jone_0} \in \Npos \\ n_{\jone_0} \gtrsim \lambda^{1/2}|x|}} n_{\jone_0}^{-\varepsilon-1} \sum_{\substack{(n_\jone)_{\jone \in J'} \in \Npos^{J'} \\ n_\jone \leq n_{\jone_0} \,\forall \jone \in J'}} \prod_{\jone \in J'} \frac{n_\jone^{-1/(1+\hpdeg)}}{n_{\jone_0}^{\hpdeg/(1+\hpdeg)}} \\
 &\qquad\times 
\int_{\prod_{\jone \in J} (0,4|\lambda^{1/2} x_\jone|)} \prod_{\jone \in J} \vu_\jone^{-1/2} |\vu_\jone - |\lambda^{1/2} x_\jone||^{-1/2} \,\prod_{\jone \in J} d\vu_\jone.
\end{split}\]
The last integral is easily seen (by rescaling) to be uniformly bounded in $\lambda$ and $x$. Similarly the inner sum in $(n_\jone)_{\jone \in J'}$ is uniformly bounded in $n_{\jone_0}$, since $\hpdeg > 0$. Finally the outer sum in $n_{\jone_0}$ converges and is majorized by $\max\{1,\lambda^{1/2} |x|\}^{-\varepsilon}$, since $\varepsilon > 0$, which makes the above quantity uniformly bounded in $\lambda$ and $x$ overall.

Suppose now instead that $\jone_0 \in J$. Then, by \eqref{eq:trans_est1} and \eqref{eq:trans_est2}, for all $n \in N_\lambda^x(J) \cap M(\jone_0)$ such that $|x|_\infty \leq \kappa |\tilde x_{n,\lambda}|_\infty$,
\[
\lambda^{1/2} |x| \lesssim \lambda^{1/2} |\tilde x_{n,\lambda}| \simeq |n|_\infty = n_{\jone_0} \simeq \lambda^{1/2} (\tilde x_{n,\lambda})_{\jone_0} \lesssim \lambda^{1/2} |x_{\jone_0}| \leq \lambda^{1/2} |x|,
\]
which shows that all these quantities are actually equivalent (and in particular $\lambda^{1/2} |x| \gtrsim 1$). Therefore in this case the right-hand side of \eqref{eq:remaining_sum_explicit} is controlled by
\[
 \sum_{\substack{(n_\jone)_{\jone \in J^c} \in \Npos^{J^c} \\ n_\jone \lesssim |\lambda^{1/2} x| \,\forall \jone \in J^c}} \prod_{\jone \in J^c} \frac{n_\jone^{-1/(1+\hpdeg)}}{|\lambda^{1/2} x|^{\hpdeg/(1+\hpdeg)}} 
\sum_{\substack{(n_\jone)_{\jone \in J} \in \Npos^J \\ (\lambda^{1/2} \tilde x_{n,\lambda})_\jone \leq 2|\lambda^{1/2} x_\jone|}}
\prod_{\jone \in J} \Phi_{\lambda^{1/2} x_\jone}((\lambda^{1/2} \tilde x_{n,\lambda})_\jone).
\]
Analogously as before, we can apply Lemma \ref{lem:sum_integral} to majorize the inner sum; so we obtain that
 the above quantity is controlled by
\[\begin{split}
& \sum_{\substack{(n_\jone)_{\jone \in J^c} \in \Npos^{J^c} \\ n_\jone \lesssim |\lambda^{1/2} x| \,\forall \jone \in J^c}} \prod_{\jone \in J^c} \frac{n_\jone^{-1/(1+\hpdeg)}}{|\lambda^{1/2} x|^{\hpdeg/(1+\hpdeg)}} \\
 &\qquad\times 
\int_{\prod_{\jone \in J} (0,4|\lambda^{1/2} x_\jone|)} \prod_{\jone \in J} \vu_\jone^{-1/2} |\vu_\jone - |\lambda^{1/2} x_\jone||^{-1/2} \,\prod_{\jone \in J} d\vu_\jone
\end{split}\]
and, similarly as before, both the inner integral and the outer sum are uniformly bounded in $\lambda$ and $x$, because $\hpdeg > 0$.
\end{proof}

Define, for all $R \in \Rpos$, the weight $\weight_R : \RR^{\done+\dtwo} \times \RR^{\done+\dtwo} \to \Rnon$ by
\[
\weight_R(z,z') = \frac{|x|}{\max\{R^{-1},|x'|\}}.
\]
for all $z=(x,y)$ and $z'=(x',y')$ in $\RR^{\done+\dtwo}$.
Combining the above pointwise estimate with Propositions \ref{prp:doubling_gaussian} and \ref{prp:weighted1}
finally yields the weighted Plancherel estimates.

\begin{cor}
Let $\gamma \in [0,\dtwo\hpdeg/2)$. For all $F : \RR \to \CC$ and $z' = (x',y') \in \RR^{\done+\dtwo}$,
\[
\| W^\gamma \Kern_{F(\opL)}(\cdot,z')\|_{L^2(\RR^{\done+\dtwo})}^2 
\lesssim_\gamma \int_0^\infty |F(\lambda)|^2 \, \lambda^{(\done+\dtwo)/2} \min\{\lambda^{1/2},|x'|^{-1}\}^{\dtwo\hpdeg-2\gamma} \, \frac{d\lambda}{\lambda}.
\]
In particular, for all $R \in \Rpos$, if $\supp F \subseteq [R^2/4,4R^2]$, then
\begin{equation}\label{eq:weighted2_final}
\Vol(z', R^{-1}) \, \| \weight_R(\cdot,z')^\gamma \Kern_{F(\opL)}(\cdot,z')\|_{L^2(\RR^{\done+\dtwo})}^2 
\lesssim_\gamma \| F(R^2 \cdot) \|_2^2.
\end{equation}
\end{cor}

\subsection{Proof of Theorem \ref{thm:main}}\label{ss:main}
We are finally able to prove our main result. We need some properties of the weight $\weight_R$. Recall that the homogeneous dimension of the underlying metric measure space is $Q = \done+(1+\hpdeg)\dtwo$.

\begin{lem}
Suppose that $\gamma \in [0,\min\{\done,\hpdeg\dtwo\}/2)$ and $\beta > Q/2-\gamma$. Then
\begin{equation}\label{eq:weight_integral}
\int_{\RR^{\done+\dtwo}} (1+\weight_R(z,z'))^{-2\gamma} (1+R\dist(z,z'))^{-2\beta} \,dz \lesssim_{\beta,\gamma} \Vol(z',R^{-1}).
\end{equation}
for all $R \in \Rpos$ and $z' \in \RR^{\done+\dtwo}$.
Moreover
\begin{equation}\label{eq:weight_dist_estimate}
\weight_R(z,z') \lesssim 1+R\dist(z,z')
\end{equation}
for all $R \in \Rpos$ and $z,z' \in \RR^{\done+\dtwo}$.
\end{lem}
\begin{proof}
Analogous to \cite[Lemma 12]{martini_grushin_2012} and \cite[Lemma 4.2]{chen_sharp_2013}.
\end{proof}

Recall that $D = \max\{\done+\dtwo,(1+\hpdeg)\dtwo\}$. By standard techniques (see, e.g., \cite[Section 5]{martini_grushin_2012} and references therein), the proof of Theorem \ref{thm:main} reduces to the following weighted $L^1$ estimate for integral kernels $\Kern_{F(\opL)}$ corresponding to compactly supported multipliers $F$.

\begin{prp}
Let $R \in \Rpos$ and $\beta,s \in \Rnon$ with $s > \beta+D/2$. Then, for all bounded Borel functions $F : \RR \to \CC$ supported in $[R^2/4,R^2]$,
\[
\esssup_{z' \in \RR^{\done+\dtwo}} \int_{\RR^{\done+\dtwo}} (1+R\dist(z,z'))^\beta \, |\Kern_{F(L)}(z,z')| \,dz \lesssim_{\beta,s} \|F(R^2 \cdot)\|_{\sobolev{2}{s}}.
\]
\end{prp}
\begin{proof}
It is convenient to introduce, for all $p \in [1,\infty]$ and $\beta,\gamma \in \Rnon$ and $r \in \Rpos$, the following weighted $L^p$ norm for integral kernels $K : \RR^{\done+\dtwo} \times \RR^{\done+\dtwo} \to \CC$:
\begin{multline*}
\vvvert K \vvvert_{p,\beta,\gamma,r} \\
= \esssup_{z' \in \RR^{\done+\dtwo}} \Vol(z',r)^{1/p'} \| (1+\dist(\cdot,z')/r)^\beta (1+\weight_{r^{-1}}(\cdot,z'))^\gamma K(\cdot,z') \|_{L^p(\RR^{\done+\dtwo})}.
\end{multline*}
Here $p' = p/(p-1)$ denotes the conjugate exponent to $p$.

By Proposition \ref{prp:doubling_gaussian}, $\RR^{\done+\dtwo}$ with the Lebesgue measure and the distance $\dist$ is a doubling space, and $\opL$ satisfies Gaussian-type heat kernel bounds, hence standard multiplier results (see, e.g., \cite{hebisch_functional_1995}, \cite{duong_plancherel-type_2002} or \cite[Theorem 6.1]{martini_crsphere}) apply to $\opL$ and yield the following weigthed $L^2$ estimate:
\[
\vvvert \Kern_{F(\opL)} \vvvert_{2,\beta,0,R^{-1}} \lesssim_{\beta,s} \| F(R^2 \cdot) \|_{\sobolev{\infty}{s}}
\]
for all $R \in \Rpos$, $F : \RR \to \CC$ with $\supp F \subseteq [-4R^2,4R^2]$ and $s > \beta \geq 0$; together with \eqref{eq:weight_dist_estimate} and Sobolev embedding, this implies that
\[
\vvvert \Kern_{F(\opL)} \vvvert_{2,\beta,\gamma,R^{-1}} \lesssim_{\beta,\gamma,s} \| F(R^2 \cdot) \|_{\sobolev{2}{s}}
\]
for all $R \in \Rpos$, $F : \RR \to \CC$ with $\supp F \subseteq [-4R^2,4R^2]$, $\beta,\gamma \in \Rnon$ and $s > \beta +\gamma+1/2$.

Note that \eqref{eq:weighted2_final} can be rewritten as
\[
\vvvert \Kern_{F(\opL)} \vvvert_{2,0,\gamma,R^{-1}} \lesssim_\gamma \| F(R^2 \cdot) \|_2
\]
for all $R \in \Rpos$, $F : \RR \to \CC$ with $\supp F \subseteq [R^2/4,4R^2]$ and $\gamma \in [0,\dtwo\hpdeg/2)$. Interpolation of the last two estimates (cf.\ \cite[proof of Lemma 1.2]{mauceri_vectorvalued_1990} or \cite[proof of Lemma 4.3]{duong_plancherel-type_2002}) yields
\[
\vvvert \Kern_{F(\opL)} \vvvert_{2,\beta,\gamma,R^{-1}} \lesssim_{\beta,\gamma,s} \| F(R^2 \cdot) \|_{\sobolev{2}{s}}
\]
for all $R \in \Rpos$, $F : \RR \to \CC$ with $\supp F \subseteq [R^2/4,4R^2]$, $\gamma \in [0,\dtwo\hpdeg/2)$ and $s > \beta \geq 0$. This estimate, together with \eqref{eq:weight_integral} and the Cauchy--Schwarz inequality implies
\[
\vvvert \Kern_{F(\opL)} \vvvert_{1,\beta,0,R^{-1}} \lesssim_s \| F(R^2 \cdot) \|_{\sobolev{2}{s}}
\]
for all $R \in \Rpos$, $F : \RR \to \CC$ with $\supp F \subseteq [R^2/4,4R^2]$, $\beta \in \Rnon$ and $s > \beta + (Q-\min\{\done,\dtwo\hpdeg\})/2 = \beta+D/2$.
\end{proof}

\section*{Appendix. Proof of Theorem \ref{titchmarsh-thm}}

Theorem $3.5$ of \cite{berezin-shubin} states that the eigenfunction $\psi_n^V$ associated to the eigenvalue $E_n^V$ has exactly $n-1$ zeros. Our task is thus reduced to counting the zeros of any nontrivial $L^2$ solution of $u''(x)=(V(x)-E_n^V)u(x)$ on $(0,+\infty)$ 
(and the analogous problem on $(-\infty,0)$). A standard argument in Sturm--Liouville theory implies that such a function $u$ \emph{cannot have zeros outside the classical region $(0,x_n^{V,+})$}: indeed, if $x_1 \geq x_n^{V,+}$ is any such zero, without loss of generality $u'(x_1)>0$ and $u(x_2), u'(x_2)>0$ if $x_2$ is slightly to the right of $x_1$; using $V(x)-E_n^V>0$ on $[x_2,+\infty)$ one can easily conclude that $u$ diverges at $+\infty$, contradicting square-integrability. Thus, it is enough to estimate
\[
Z \defeq \text{number of zeros of $u$ on $(0,x_n^{V,+})$}.
\] 
In order to do that, we introduce the additional point $y_n^{V,+}$ defined as the unique solution of $\frac{V'(y)}{(E_n^V-V(y))^{3/2}}=\frac{1}{\pi}$ in $(0,x_n^{V,+})$. By Sturm's comparison theorem \cite[Section 5.2]{titchmarsh}, the number of zeros of $u$ on $[y_n^{V,+},x_n^{V,+})$ is bounded by the maximal number of zeros of a nontrivial solution of $v''(x)=(V(y_n^{V,+})-E_n^V)v(x)$ on the same interval. Since
\begin{equation}\label{titchmarsh-diff}
x_n^{V,+}-y_n^{V,+}\leq \frac{1}{V'(y_n^{V,+})}\int_{y_n^{V,+}}^{x_n^{V,+}}V' = \frac{E_n^V-V(y_n^{V,+})}{V'(y_n^{V,+})} = \frac{\pi}{\sqrt{E_n^V-V(y_n^{V,+})}},
\end{equation}
such a solution $v$ can have at most one zero in $[y_n^{V,+},x_n^{V,+})$. Thus $\left|Z-Z'\right|\leq 1$, where
\[
Z' \defeq \text{number of zeros of $u$ on $(0,y_n^{V,+})$}.
\]

Now define $Q(x) \defeq \sqrt{E_n^V-V(x)}$ and $\rho(x) \defeq \sqrt{u'(x)^2 + Q(x)^2u(x)^2}$ for all $x\in(0,x_n^{V,+})$, and let $\theta:(0,x_n^{V,+})\rightarrow\RR$ be a $C^1$ function such that
\begin{equation}\label{titchmarsh-theta}
e^{i\theta} = \frac{u'+ iQu}{\rho}.
\end{equation}
The fact that the phase function $\theta$ is defined up to integer multiples of $2\pi$ has no consequences in what follows. Differentiating \eqref{titchmarsh-theta} and taking real parts yields
\begin{equation}\label{titchmarsh-theta'}
\theta' = Q + \frac{Q'uu'}{\rho^2}.
\end{equation}
By combining \eqref{titchmarsh-theta} and \eqref{titchmarsh-theta'}, we see that $u(x)=0$ if and only if $\theta(x)\in \pi\ZZ$, and in those points $\theta'(x)>0$. This immediately implies that
\[
\left|Z'- \frac{1}{\pi}\int_0^{y_n^{V,+}}\theta'\right|\leq1.
\]
Using \eqref{titchmarsh-theta'} again, we find
\[
\left|Z'- \frac{1}{\pi}\int_0^{x_n^{V,+}}Q\right|\leq1 + \frac{1}{\pi}\int_{y_n^{V,+}}^{x_n^{V,+}}Q + \left|\frac{1}{\pi}\int_0^{y_n^{V,+}}\frac{Q'uu'}{\rho^2}\right|.
\]
Estimate \eqref{titchmarsh-diff} gives $\frac{1}{\pi}\int_{y_n^{V,+}}^{x_n^{V,+}}Q\leq 1$. By squaring and taking imaginary parts of both sides of \eqref{titchmarsh-theta}, we obtain the identity 
\begin{equation}\label{titchmarsh-im}
\frac{Q'uu'}{\rho^2}=\frac{\sin(2\theta)Q'}{2Q}.
\end{equation}
Putting our estimates together gives
\begin{equation}\label{titchmarsh-intermediate}
\left|Z- \frac{1}{\pi}\int_0^{x_n^{V,+}}Q\right|\leq  3 + \left|\frac{1}{\pi}\int_0^{y_n^{V,+}}\frac{\sin(2\theta)Q'}{2Q}\right|.
\end{equation}
To bound the last integral, we combine \eqref{titchmarsh-theta'} and \eqref{titchmarsh-im} to get
\[
\int_0^{y_n^{V,+}}\frac{\sin(2\theta)Q'}{2Q} = \int_0^{y_n^{V,+}}\frac{\sin(2\theta)Q'}{2Q^2}\theta' - \int_0^{y_n^{V,+}}\frac{\sin(2\theta)^2|Q'|^2}{4Q^3}.
\]
We treat the two integrals separately. We have
\[
\int_0^{y_n^{V,+}}\frac{\sin(2\theta)Q'}{2Q^2}\theta'  = \frac{1}{4}\int_0^{y_n^{V,+}}(\cos(2\theta))'\left(\frac{1}{Q}\right)',
\]
where $\left(\frac{1}{Q}\right)'=\frac{V'}{2(E_n^V-V)^{3/2}}$ is monotone and takes values in $(0,\frac{1}{2\pi})$ on the interval of integration. By the second mean value theorem for Riemann--Stieltjes integrals, there exists $\bar{x}\in (0,y_n^V)$ such that
\[
\frac{1}{4}\left|\int_0^{y_n^{V,+}}(\cos(2\theta))'\left(\frac{1}{Q}\right)'\right|=\frac{1}{4}\left|\frac{1}{2\pi}\int_{\bar{x}}^{y_n^{V,+}}(\cos(2\theta))'\right|\leq \frac{1}{4\pi}.
\]
The second integral is bounded as follows:
\[\begin{split}
0
&\leq \int_0^{y_n^{V,+}}\frac{\sin(2\theta)^2|Q'|^2}{4Q^3}\leq \int_0^{y_n^{V,+}}\frac{|Q'|^2}{4Q^3}=\frac{1}{16}\int_0^{y_n^{V,+}}\frac{V'}{(E_n^V-V)^{\frac{5}{2}}}V'\\
&= \frac{1}{24}\int_0^{y_n^{V,+}}((E_n^V-V)^{-{3/2}})' \, V'\\
&= -\frac{1}{24}\int_0^{y_n^{V,+}}(E_n^V-V)^{-{3/2}} \, V'' + (E_n^V-V(y_n^{V,+}))^{-{3/2}} \,V'(y_n^{V,+})\\
&\leq \frac{1}{\pi},
\end{split}\]
where we used an integration by parts. Altogether, we proved
\[
\frac{1}{\pi}\left|\int_0^{y_n^{V,+}}\frac{\sin(2\theta)Q'}{2Q}\right|\leq \frac{5}{4\pi^2}.
\]
Recalling the definition of $Z$ and \eqref{titchmarsh-intermediate}, we have
\[
\left|\text{number of zeros of $\psi_n^V$}-\frac{1}{\pi}\int_{-x_n^{V,+}}^{x_n^{V,+}}\sqrt{E_n^V-V}\right|\leq 2\left(3 + \frac{5}{4\pi^2}\right)+1, 
\]
as we wanted (the summand $1$ comes from the possible zero at $x=0$).


\providecommand{\bysame}{\leavevmode\hbox to3em{\hrulefill}\thinspace}

\end{document}